\documentclass{amsart}
\usepackage{amssymb}
\usepackage{amscd}
\usepackage{amsmath}
\usepackage{epsfig}
\usepackage{graphicx}
\usepackage{color}

\usepackage{hyperref}

\usepackage{time}

\usepackage[all]{xy}

\usepackage{pgf,tikz}
\usepackage{caption}
\usepackage{shuffle}

\title[$\ell$-adic hypergeometric function]
{The $\ell$-adic hypergeometric function and associators}
\author{Hidekazu Furusho}
\address{Graduate School of Mathematics, Nagoya University, Chikusa-ku, Furo-cho, Nagoya, 464-8602,  Japan}
\email{furusho@math.nagoya-u.ac.jp}



\date{September 1, 2022}
\newtheorem{thm}{Theorem}[section]
\newtheorem{lem}[thm]{Lemma}

\newtheorem{prop}[thm]{Proposition}  

\theoremstyle{remark}

\theoremstyle{definition}
\newtheorem{defn}[thm]{Definition}
\newtheorem{rem}[thm]{Remark}
\newtheorem{nota}[thm]{Notation}

\numberwithin{equation}{section}
\numberwithin{figure}{section}

\newcommand{\dch}{\mathrm{dch}}

\newcommand{\GL}{\mathrm{GL}}
\newcommand{\SL}{\mathrm{SL}}
\newcommand{\ev}{\mathrm{ev}}
\newcommand{\sw}{\mathrm{sw}}
\newcommand{\GT}{\mathrm{GT}}

\newcommand{\Q}{\mathbb{Q}}
\newcommand{\C}{\mathbb{C}}
\newcommand{\R}{\mathbb{R}}
\newcommand{\Z}{\mathbb{Z}}
\newcommand{\N}{\mathbb{N}}
\newcommand{\K}{\mathbb{K}}

\newcommand{\V}{\mathcal{V}}
\newcommand{\X}{\mathcal{X}}
\newcommand{\PP}{\mathcal{R}}

\newcommand{\aaa}{\mathsf a}
\newcommand{\bbb}{\mathsf b}
\newcommand{\ccc}{\mathsf c}
\newcommand{\ppp}{\mathsf p}
\newcommand{\qqq}{\mathsf q}

\newcommand{\wt}{\mathrm{wt}}

\newcommand{\Mat}{\mathrm{Mat}}
\newcommand{\Ad}{\mathrm{Ad}}

\newcommand{\comp}{\mathrm{comp}}

\newcommand*{\HG}[3]{{}_{2}F_{1}{\left({{#1}\atop{#2}}\middle|{#3}\right)}}
\newcommand*{\HGdag}[3]{{}_{2}F_{1}^\dagger{\left({{#1}\atop{#2}}\middle|{#3}\right)}}
\newcommand*{\HGddag}[3]{{}_{2}F_{1}^\ddagger{\left({{#1}\atop{#2}}\middle|{#3}\right)}}
\newcommand*{\dHG}[3]{{}_{2}F'_{1}{\left({{#1}\atop{#2}}\middle|{#3}\right)}}

\newcommand*{\GGamma}[3]{\Gamma_{#1}{\left({{#2}\atop{#3}}\right)}}

\catcode`,\active

\catcode`\,12

\usepackage{ulem}
\newcommand{\Add}[1]{\textcolor{black}{#1}}

\begin{document}
\bibliographystyle{amsalpha+}
\maketitle

%%%%%%%%%%%%%%%%%%%%%%%%%%%%%%%%%%%%%%%%%%%%%%%%%%%%%%%%%%%%%%%%%%%%%%
\begin{abstract}
%$\HGddag{\aaa,\bbb}{\ccc}{z}$
We introduce an $\ell$-adic analogue of Gauss's hypergeometric function
arising from the Galois action on
the fundamental torsor of the projective line minus  three points.
Its definition is motivated by a relation between
the KZ-equation and the hypergeometric differential equation in the complex case.
We show two basic properties,
analogues of Gauss's hypergeometric theorem and
of Euler's transformation formula
%and Pfaff's one
for our $\ell$-adic function.
We prove them by
detecting a connection of
a certain two-by-two matrix  specialization of  even unitary associators 
with
the associated gamma function, which extends the result of Ohno and Zagier.
\end{abstract}

\tableofcontents
%%%%%%%%%%%%%%%%%%%%%%%%%%%%%%%%%%%%%%%%%%%%%%%%%%%%%%%%%%%%%%%%%%%%%%%
\setcounter{section}{-1}
\section{Introduction}\label{introduction}
%%%%%%%%%%%%%%%%%%%%%%%%%%%%%%%%%%%%%%%%%%%%%%%%%%%%%%%%%%%%%%%%%%%%%%%
%\section{$\ell$-adic hypergeometric function}
It is said that the hypergeometric function first appeared in a book by Wallis (1655).
Since then, the  
hypergeometric function has attracted and widespread attention in
various areas of mathematics.
%for many centuries. 
%So far several variants have been introduced. 
In this paper we introduce a new variant,
the {\it $\ell$-adic hypergeometric function} ($\ell$: an odd prime).
It is an $\ell$-adic function parametrized by the absolute Galois group,
% of the rationals $\Q$.
which could be regarded as an \lq $\ell$-adic %\'{e}tale 
Galois avatar' of the 
hypergeometric function.
This is not the finite analogue of  hypergeometric functions
considered the literature
(see \cite{FLRST, Greene, Katz, Otsubo}, etc).
%which might be another \'{e}tale  avatar.
It also differs from 
Dwork's $p$-adic hypergeometric function (\cite{Dw}),
which is rather  a \lq crystalline avatar'.
However, one might expect any intimate relationship with these functions
by discussing any possibly common motives.

%Our work is motivated by the following:
Our construction of the $\ell$-adic hypergeometric function
is motivated by the following:
\begin{itemize}
\item
In the $\ell$-adic \'{e}tale setting,
Wojtkowiak (\cite{W}) and Nakamura-Wojtkowiak (\cite{NW})
introduced and explored
the {\it $\ell$-adic %(multiple) 
polylogarithm}, %an avatar in the \'{e}tale side of 
which is associated with
the $\ell$-adic Galois representation on the fundamental torsor of the projective line minus  three points.
It is an $\ell$-adic function parametrized by the absolute Galois group
and topological paths.
It is considered as an $\ell$-adic Galois avatar of
Coleman's $p$-adic polylogarithm (\cite{Col}).
% (and its multiple version \cite{F04}).
\item
In the complex case, 
intimate relations between the KZ-equation and
the hypergeometric differential equation have been discussed in the
 literature; see \cite{SV}, \cite{V}, etc.
%Multiple polylogarithms appear as coefficients of the fundamental solution of the KZ-equation (see \cite{??} for exact formulae). 
In particular, Oi presented a clear formulation
in \cite{O},
where 
by reconstructing the hypergeometric function from the fundamental solution
\eqref{eq:G01}
of the KZ-equation %with coefficients of multiple polylogarithms
(cf. \eqref{eq:HG=11}),
he  deduced various relations among multiple polylogarithms 
%and multiple zeta values
appearing as its coefficients.
\end{itemize}
Our strategy is to (i) consider
the $\ell$-adic Galois cocycle  $f_\sigma^z$
(see \eqref{eq: Galois 1-cocycle} below) 
associated with the same $\ell$-adic Galois representation,
%on the fundamental torsor of the projective line minus  three points,
%Our construction of $\ell$-adic hypergeometric function is to 
(ii) regard  its image
$G_{\vec{01}}^{\varphi}(e_0,e_1)(\sigma)(z)$
%of the $\ell$-adic Galois cocycle  $f_\sigma^z$
%(see \eqref{eq: Galois 1-cocycle} below) 
%and regard its image 
under the fake comparison isomorphism (Definition \ref{defn: fake comparison})
constructed by associators
as an $\ell$-adic analogue of the fundamental solution
and  (iii) then extract  the $\ell$-adic hypergeometric function 
as the $(1,1)$ entry of the $2\times 2$-matrix
%(see Definition \ref{defn:lHG})
%in the same fashion of 
in the same way as
\eqref{eq:HG=11}:

%Our construction of $\ell$-adic hypergeometric function
%is based on  the $\ell$-adic Galois representation on the 
%the fundamental torsor of the projective line minus  three points
%in the same spirit of Wojtkowiak (\cite{W}) and Nakamura-Wojtkowiak (\cite{NW})
%who introduced and explored
%the {\it $\ell$-adic (multiple) polylogarithm}, %an avatar in the \'{e}tale side of 
%an $\ell$-adic Galois avatar of
%Coleman's $p$-adic polylogarithm (\cite{Col}) and its multiple version (\cite{F04}),
%which is also an $\ell$-adic function parametrized by the absolute Galois group
%and topological paths.
%%in crystalline side.

(i). Let $\bar\Q$ be the algebraic closure of the rational number field $\Q$
and $G_\Q:=\mathrm{Gal}(\bar \Q/\Q)$ be the absolute Galois group.
We fix an embedding $\bar \Q\hookrightarrow \C$.
We consider the algebraic curve  $\X:={\mathbb P}^1\setminus\{0,1,\infty\}$ over $\Q$,
the projective line minus the three points.
The topological fundamental group  $\pi_1^\mathrm{top}(\X(\C),\vec{01})$
of its associated topological space $\X(\C)$ with the tangential basepoint
$\vec{01}$ (cf. \cite{De}) is identified with the free group $F_2$
with the standard generators $x_0$, $x_1$, $x_\infty$ corresponding to the  loops around $0$, $1$ and $\infty$ such that $x_0x_1x_\infty=1$.
%with two generators $x$ and $y$,
%where $x$ (resp. $y$)
%corresponds to the homotopy loop going around 
%$0$ counterclockwise
%(resp. the loop  which goes along  'the straight path' $\dch_{0,1}$  from $0$ to $1$, then
%goes around $1$ counterclockwise and finaly goes back along $\dch_{0,1}$).

Let $\ell$ be an \Add{odd} prime. Let $z$ be a rational (or tangential base) point of $\X$.
We denote by $\hat\pi_1^\ell(\X_{\bar \Q};\vec{01},z)$
the profinite set of pro-$\ell$ \'{e}tale paths from $\vec{01}$ to $z$
(cf. \cite{SGA1}).
By the comparison isomorphisms induced by  the fixed embedding  $\bar \Q\hookrightarrow \C$,
we identify the pro-$\ell$ \'{e}tale fundamental group
$\hat\pi_1^\ell(\X_{\bar \Q};\vec{01}):=\hat\pi_1^\ell(\X_{\bar \Q};\vec{01},\vec{01})$
with the pro-$\ell$ completion $\hat F_2^{\ell}$ of  $F_2$,
and we regard each topological path  $\gamma_z:\vec{01}\leadsto z$
(that signifies $\gamma_z\in\pi_1^\mathrm{top}(\X(\C);\vec{01},z)$)
%from $\vec{01}$ to $z$
as a pro-$\ell$ \'{e}tale path 
$\gamma_z\in\hat\pi_1^\ell(\X_{\bar \Q};\vec{01},z)$.
Since $\X$ and $z$ are defined over $\Q$,
the set $\hat\pi_1^\ell(\X_{\bar \Q};\vec{01},z)$ admits  the action of 
$G_\Q$. % (consult \cite{SGA1} in detail).
For each $\sigma\in G_\Q$, we consider the Galois 1-cocyle
\begin{equation}\label{eq: Galois 1-cocycle}
f^z_\sigma=f^{\gamma_z}_\sigma:=%f^{z}_\sigma(x_0,x_1):=
\gamma_z^{-1}\sigma(\gamma_z)\in \hat F_2^\ell.
\end{equation}

(ii). We take an $\ell$-adic even unitary associator $\varphi$ (Definition \ref{defn:associator}).
It provides the fake comparison isomorphism $\comp_\varphi^{\vec{01}}$
 (Definition \ref{defn: fake comparison}).
Under the map  $\iota_\varphi$ in \eqref{eq:iota-varphi},
the restriction of $\comp_\varphi^{\vec{01}}$
by the inclusion $\hat F_2^{(\ell)}\hookrightarrow F_2(\Q_\ell)$,
our $\ell$-adic analogue of the fundamental solution is  defined to be the
noncommutative formal power series
$$
G_{\vec{01}}^{\varphi}(e_0,e_1)(\sigma)(z):=\iota_{\varphi}(f_\sigma^z)
\in \Q_\ell\langle\langle e_0,e_1 \rangle\rangle.
$$

(iii). Let $\aaa,\bbb,\ccc$ be variables.
Put $\ppp=1-\ccc$, $\qqq=\aaa+\bbb+1-\ccc=\aaa+\bbb+\ppp$ and
$$
X=
\begin{pmatrix}
0 & \bbb \\
0 & \ppp
\end{pmatrix}, \
Y=
\begin{pmatrix}
0 & 0 \\
\aaa & \qqq
\end{pmatrix}\in \Mat_2(\Q_\ell[\aaa,\bbb,\ccc-1]).
$$
Following \eqref{eq:HG=11},
we define the formal version of $\ell$-adic hypergeometric function to be
the $(1,1)$ entry of the above $G_{\vec{01}}^{\varphi}(e_0,e_1)(\sigma)(z)$
with substitution at $e_0=X$ and $e_1=-Y$: 
\begin{equation}\label{eq:HG and even associator}
\HG{\aaa,\bbb}{\ccc}{z}(\sigma)=
\HG{\aaa,\bbb}{\ccc}{\gamma_z}(\sigma)
:=[G_{\vec{01}}^{\varphi}(X,-Y)(\sigma)(z)
%\iota_\varphi(f_\sigma^z)(X,-Y)
]_{(1,1)}
\in\Q_\ell[[\aaa,\bbb,\ccc-1]]
\end{equation}
(another formulation is given in Proposition \ref{prop: HG=Theta}).
%In precise, it should be denoted by 
%$\HG{\aaa,\bbb}{\ccc}{\gamma_z}(\sigma)$
%because it is constructed from
%$f_\sigma^z=f_\sigma^{\gamma_z}$ which depends not only on $z$ 
%but also on $\gamma_z$.
%However we simply denote $\HG{\aaa,\bbb}{\ccc}{z}(\sigma)$
%by abuse of notation as
%the $\ell$-adic $m$-th polylogarithm was simply denoted by 
%$\ell_m^z$ in \cite{NW}. 

Our  first result is on the well-definedness of the
$\ell$-adic hypergeometric function. 

\begin{thm}\label{fundamental theorem of l-adic HG}
(i).
The definition of $\HG{\aaa,\bbb}{\ccc}{z}(\sigma)$ is independent of any choice of  $\ell$-adic even unitary associator  $\varphi$.

(ii).
Suppose that $a,b,c\in\Z_\ell$ 	are with
\begin{equation}\label{eq:ABC condition}
|a|_\ell,|b|_\ell,|c-1|_\ell<1.
\end{equation}
Then, $\HG{a,b}{c}{z}(\sigma)$ converges.
\end{thm}

So our $\ell$-adic  hypergeometric function may be regarded to be the map
$$
\HG{a,b}{c}{-}(\sigma):
\underset{z}{\amalg}\ \pi_1^\mathrm{top}({\mathcal X}(\C);{\vec{01}},z)
\to\Q_\ell
$$
sending $\gamma_z\mapsto\HG{a,b}{c}{\gamma_z}(\sigma)$
where $z$ runs over  rational (may tangential base) 		points of $X$.

Our second main theorem  is an analogue of the Gauss hypergeometric theorem 
\eqref{eq: hypergeometric equation}. % due to Gauss.

\begin{thm}[$\ell$-adic Gauss hypergeometric theorem]
\label{thm: l-adic Gauss's hypergeometric theorem}
Take $a,b,c\in\Z_\ell$ which satisfy \eqref{eq:ABC condition}.
Put $p=1-c$, $q=a+b+1-c=a+b+p$.
For $\sigma\in G_{\Q}$, the following equality holds:
{\small
\begin{align}
&\HG{a,b}{c}{\vec{10}}(\sigma)=
\frac{pq}{ab}\left\{\GGamma{+}{-p,-q}{-p-a,-p-b}+\GGamma{+}{p,-q}{-a,-b}\right\}\GGamma{+}{-p,q}{a,b}\cdot\GGamma{\sigma}{-p,q}{a,b}+ \\ \notag
&\left\{\frac{ab+pq}{pq}\GGamma{+}{p,q}{p+a,p+b}
-\frac{ab}{pq}\GGamma{+}{-p,q}{a,b}\right\}
\GGamma{+}{-p,-q}{-p-a,-p-b}\cdot
\GGamma{\sigma}{-p,-q}{-p-a,-p-b}.
\end{align}
}
%\end{thm}
Here, the path $\gamma:\vec{01}\leadsto\vec{10}$ is chosen to be the straight path (often denoted by $\dch$ in the literarure),
%from $0$ to $1$ 
%in $\hat\pi_1^\ell(X_{\bar \Q};\vec{01},\vec{10})$ 
$\Gamma_\sigma$ is defined by 
the $\ell$-adic series (cf. Remark \ref{rem:Gamma})
related to the hyperadelic gamma function of Anderson (\cite{A})
and $\Gamma_+$ is defined by the \lq +'-part of the classical gamma function
(cf. Notation \ref{nota: M+}).
\end{thm}
The condition \eqref{eq:ABC condition}
for $a$, $b$, $c$ in $\HG{a,b}{c}{z}(\sigma)$
will be relaxed to $(a,b,c)$ in
$$
\mathfrak D:=
%\amalg_{i=1}^3\{P_i+(\ell \Z_\ell)^3\}
\{(a,b,c)\in\Z_\ell^2\times \Z_\ell^\times \bigm|
(a,b,c)\equiv
(0,0,1), (0,0,0)\text{ or }(0,1,1) \bmod \ell %, c\neq 0
\}
$$
%with $P_1=(0,0,1)$, $P_2=(0,0,0)$, $P_3=(1,0,1)$
in Proposition \ref{prop:enlarge}. 
Our third main theorem  is an analogue of  Euler's transformation formula
\eqref{eq: Euler's transformation}.

\begin{thm}[$\ell$-adic Euler transformation formula]
\label{thm: l-adic Euler's transformation formula}
Let $a,b,c\in\ell\Z_\ell$ with $c\neq 0$.
%satisfying $|a|_\ell,|b|_\ell,|c|_\ell<1$.
%Let $(a,b,c)\in\mathfrak D$.
Let $z$ be a rational or tangential base point of $\X$.
Then, for $\sigma\in G_\Q$, we have the equality
\begin{equation}
\HG{a,b}{c}{z}(\sigma)=%(1-z)^{c-a-b}
\exp\left\{(c-a-b)\rho_{1-z}(\sigma)\right\}\cdot
\HG{c-a,c-b}{c}{z}(\sigma)
\end{equation}
%with $(1-z)^{c-a-b}=f^z_\sigma(1,\exp\{c-a-b\})\in\Z_\ell^\times$.
%with $(1-z)^{c-a-b}=\exp\{(c-a-b)\rho_{1-z}(\sigma)\}$
with  the Kummer 1-cocycle $\rho_{1-z}:G_\Q\to \Z_\ell$ %(cf. \cite{NW}) 
defined by 
$\sigma((1-z)^\frac{1}{\ell^n})=
\zeta_{\ell^n}^{\rho_{1-z}(\sigma)}(1-z)^\frac{1}{\ell^n}$ for $n\in\N$
where $\zeta_{\ell^n}=\exp\{\frac{2\pi\sqrt{-1}}{\ell^n}\}$ and the $\ell^n$-th root $(1-z)^\frac{1}{\ell^n}$ is chosen along $\gamma_z$
with $1^\frac{1}{\ell^n}=1$ at $z=0$.
\end{thm}

%Our third main theorem  is an analogue of  Pfaff's transformation formula
%\eqref{eq: Pfaff's transformation}.
%
%\begin{thm}[$\ell$-adic Pfaff transformation formula]
%\label{thm: l-adic Pfaff transformation formula}
%Under the same condition of the  above theorem,
%%Let $a,b,c\in\ell\Z_\ell$ with $c\neq 0$.
%%%satisfying $|a|_\ell,|b|_\ell,|c|_\ell<1$.
%%%Let $(a,b,c)\in\mathfrak D$.
%%Let $z$ be a rational or tangential base point of $\X$.
%%Then for $\sigma\in G_\Q$  
%we have the equality
%\begin{equation}
%\HG{a,b}{c}{z}(\sigma)=%(1-z)^{c-a-b}
%\exp\left\{-a\cdot\rho_{1-z}(\sigma)\right\}\cdot
%\HG{c-b,a}{c}{\frac{z}{z-1}}(\sigma)
%\end{equation}
%%%with $(1-z)^{c-a-b}=f^z_\sigma(1,\exp\{c-a-b\})\in\Z_\ell^\times$.
%%%with $(1-z)^{c-a-b}=\exp\{(c-a-b)\rho_{1-z}(\sigma)\}$
%%with  the Kummer 1-cocycle $\rho_{1-z}:G_\Q\to \Z_\ell$ %(cf. \cite{NW}) 
%%defined by 
%%$\sigma((1-z)^\frac{1}{\ell^n})=
%%\zeta_{\ell^n}^{\rho_{1-z}(\sigma)}(1-z)^\frac{1}{\ell^n}$ for $n\in\N$
%%where $\zeta_{\ell^n}=\exp\{\frac{2\pi\sqrt{-1}}{\ell^n}\}$ and the $\ell^n$-th root $(1-z)^\frac{1}{\ell^n}$ is chosen along $\gamma_z$
%%with $1^\frac{1}{\ell^n}=1$ at $z=0$.
%\end{thm}

The above two theorems are based on the computations of Oi (\cite{O})
and %the following result on associators (cf. Definition \ref{defn:associator}).
the following key result %(Theorem \ref{thm: varphi=M}) 
on associators.

\begin{thm}\label{thm:even unitary associator=M+0}
%Let $a,b,c\in\Z_\ell$ satisfying \eqref{eq:ABC condition}.
The following equality holds for
any even unitary associator $\varphi$ (see Definition \ref{defn:associator}):
$$\varphi(X,-Y)=M_{+}$$
where $M_{+}$ is the matrix  defined in Notation \ref{nota: M+}.
\end{thm}

The above theorem  is derived from  Theorem \ref{thm: varphi=M} 
%by a specialization,
where we observe the Ohno-Zagier relation \cite{OZ}
in the $(1,1)$ entry.
Note  that
similar but different  matrix specializations of $f_\sigma^{z}$
with $z=\vec{10}$, $1/2$, $e^{2\pi i/6}$
and their connections with various arithmetic invariants
are investigated in \cite{N,NT,NTY}.

%For associators, see Definition \ref{defn:associator}.
The remainder of this paper proceeds as follows:
In \S \ref{sec:complex case}, we review 
how the fundamental solutions of the Gauss hypergeometric differential equation
are obtained from those of the KZ equation.
The arguments are transformed to the $\ell$-adic setting,  and
%which play a preliminary role 
%in our construction of 
%We introduce 
$\ell$-adic variants of the solutions 
are constructed
in \S \ref{sec:l-adic 6 solutions}.
We prove the fundamental theorem 
(Theorem \ref{fundamental theorem of l-adic HG}), 
the $\ell$-adic Gauss hypergeometric theorem
(Theorem \ref{thm: l-adic Gauss's hypergeometric theorem}),
the above key theorem (Theorem \ref{thm:even unitary associator=M+0}) 
%are proved 
in \S \ref{sec:Gauss's hypergeometric theorem},
and then
the $\ell$-adic Euler transformation formula
(Theorem \ref{thm: l-adic Euler's transformation formula})
is proved 
in \S \ref{sec:Euler's transformation formula}.

%%%%%%%%%%%%%%%%%%%%%%%%%%%%%%%%%%%%%%%%%%%%%%%%%%%%%%%%%%%%%%%%%%%%%%%
\section{KZ equation and the hypergeometric function}\label{sec:complex case}
In this section, we recall the basic properties of the fundamental solutions of
the  KZ equation and of the hypergeometric equation (cf. \cite[\S 4.2]{EFK}),
which  play a role in the subsequent sections.

The (formal) {\it KZ (Knizhnik-Zamolodchikov) equation} is the differential equation
$$
\frac{d}{dz}G(z)=\{\frac{e_0}{z}+\frac{e_1}{z-1}\}\cdot G(z)
$$
where $G(z)$ is analytic (that is, each coefficient is analytic) in complex variables with values in 
the noncommutative formal power series ring 
$\C\langle\langle e_0,e_1\rangle\rangle$.
It has singularities at $z=0$, $1$ and $\infty$
that are  regular  and Fuchsian.
In \cite{Dr89}\S 3, Drinfeld  considers the fundamental solution 
\begin{equation}\label{eq:G01}
G_{\vec{01}}(e_0,e_1)(z)
\end{equation}
with the asymptotic property
$G_{\vec{01}}(e_0,e_1)(z)z^{-e_0}\to 1$ when $z\in\R_+$ approaches $0$.
Multiple polylogarithms appear as its coefficients (cf. \cite{F03,F04}).
He further investigates 5 other solutions %of the equation
with certain specific asymptotic properties,
which are described as
%\Add{
\begin{align}\label{eqB}
G_{\vec{10}}(e_0,e_1)(z)
&=G_{\vec{01}}(e_1,e_0)(1-z)
%=G_{\vec{1\infty}}(e_0,e_1)(z)\cdot\exp\{\pi\sqrt{-1}e_1\}
, \\ \notag
G_{\vec{1\infty}}(e_0,e_1)(z)
&=G_{\vec{01}}(e_1,e_\infty)(1-\frac{1}{z})
%=G_{\vec{\infty1}}(e_0,e_1)(z)\cdot\Phi_{\mathrm KZ}(e_1,e_\infty)
, \\ \notag
G_{\vec{\infty 1}}(e_0,e_1)(z)
&=G_{\vec{01}}(e_\infty,e_1)(\frac{1}{z})
%=G_{\vec{\infty 0}}(e_0,e_1)(z)\cdot\exp\{\pi\sqrt{-1}e_\infty\}
, \\ \notag
G_{\vec{\infty 0}}(e_0,e_1)(z)
&=G_{\vec{01}}(e_\infty,e_0)(\frac{1}{1-z})
%=G_{\vec{0\infty}}(e_0,e_1)(z)\cdot\Phi_{\mathrm KZ}(e_\infty,e_0)
, \\ \notag
G_{\vec{0\infty}}(e_0,e_1)(z)
&=G_{\vec{01}}(e_0,e_\infty)(\frac{z}{z-1})
%=G_{\vec{01}}(e_0,e_1)(z)\cdot\exp\{\pi\sqrt{-1}e_0\}
\end{align}
%}
with $e_\infty=-e_0-e_1$
under appropriate choice of branches.
%Here $e_\infty=-e_0-e_1$ and
% $\Phi_{\mathrm KZ}=\Phi_{\mathrm KZ}(e_0,e_1)
%\in \C\langle\langle e_0,e_1\rangle\rangle$
%is
%the {\it KZ-associator}, 
%the generating series of multiple zeta values
%defined by the quotient
%\begin{equation}\label{eq:KZ associator}
%G_{\vec{01}}(e_0,e_1)(z)=
%G_{\vec{10}}(e_0,e_1)(z)\cdot
%\Phi_{\mathrm KZ}(e_0,e_1).
%\end{equation}

{\it Gauss's hypergeometric function} (consult \cite{AAR} for example)
is the complex analytic function 
defined by the power series
$$
\HG{a,b}{c}{z}:=\sum_{n=0}^\infty\frac{(a)_n(b)_n}{(c)_n n!}z^n
$$
which converges for $|z|<1$.
Here, $a,b,c$ are complex numbers with $c\neq 0,-1,-2,\cdots$, and
$(s)_0=1$, $(s)_n=s(s+1)\cdots (s+n-1)$ are the Pochhammer symbols.
%Various formulae of the hypergeometric functions have been discovered  so far.
The hypergeometric function has %been studied for several centuries.
a rich history.
The contributions of  Euler and Gauss are significant, and 
in particular, 
the following  identities might be the most celebrated:
\begin{itemize}
\item
Hypergeometric theorem (due to Gauss)
%is on the special (limit) value
\begin{equation}\label{eq: hypergeometric equation}
\HG{a,b}{c}{1}=\frac{\Gamma(c)\Gamma(c-a-b)}{\Gamma(c-a)\Gamma(c-b)},
%\qquad
%\text{ when }
%\Re(c)>\Re(a+b),
\end{equation}
\item
Transformation theorem (due to Euler)
\begin{equation}\label{eq: Euler's transformation}
\HG{a,b}{c}{z}=(1-z)^{c-a-b}\HG{c-a,c-b}{c}{z}.
\end{equation}
%\item
%Transformation theorem (due to Pfaff)
%\begin{equation}\label{eq: Pfaff's transformation}
%\HG{a,b}{c}{z}=(1-z)^{-a}\HG{c-b,a}{c}{\frac{z}{z-1}}.
%\end{equation}
\end{itemize}

The hypergeometric function is a solution of {\it Euler's hypergeometric differential equation}:
$$
z(1-z)\frac{d^2w}{dz^2}+\{c-(a+b+1)z\}\frac{dw}{dz}-abw=0,
$$
which allows the analytic continuation of the function along a topological path
starting from $0$ to any $z$ in $\X(\C)$.
%There are  the so-called Kummer 24 solutions of %Euler's hypergeometric 
%the differential equation.
%We restrict the following 6 local solutions among them:
%\begin{align*}
%w_{\vec{01}}(z)&=\HG{a,b}{c}{z}, \\
%w_{\vec{\infty 1}}(z)&=z^{-a}\HG{a+1-c,a}{a-b+1}{\frac{1}{z}}, \\
%w_{\vec{10}}(z)&=\HG{b,a}{a+b+1-c}{1-z}, \\
%w_{\vec{1\infty}}(z)&=z^{-a}\HG{a,a+1-c}{a+b-c+1}{1-\frac{1}{z}}, \\
%w_{\vec{0\infty}}(z)&=(1-z)^{-a}\HG{c-b,a}{c}{\frac{z}{z-1}}, \\
%w_{\vec{\infty 0}}(z)&=(1-z)^{-a}\HG{c-b,a}{a-b+1}{\frac{1}{1-z}},
%\end{align*}
The %above %hypergeometric  
differential equation can be reformulated as 
\begin{equation*}
\frac{d}{dz}\vec{v}%\begin{pmatrix} v_1 \\ v_2\end{pmatrix}
=\left\{\frac{1}{z}X_0+\frac{1}{1-z}Y_0\right\}\cdot \vec{v}
%\begin{pmatrix} v_1 \\ v_2\end{pmatrix}
\end{equation*}
where we define
${X_0}=
\begin{pmatrix}
0 & b \\
0 & u
\end{pmatrix}$ and
${Y_0}=
\begin{pmatrix}
0 & 0 \\
a & v
\end{pmatrix}
\in\Mat_2(\C)
$ with
$u=1-c$ and $v=a+b+1-c$ and, put
$\vec{v}=\vec{v}(w)=
\begin{pmatrix}
w \\ \frac{z}{b}\frac{dw}{dz}
\end{pmatrix}\in\C^2$
%=(w,\ \frac{z}{b}\frac{dw}{dz})^{t}\in\C^2$.
% $v_1=w$ and $v_2=\frac{z}{b}\frac{dw}{dz}$ 
when $b\neq 0$.
%with $\vec{v}=\vec{v}(w)=(w,\ \frac{z}{b}\frac{dw}{dz})^{t}\in\C^2$.
% $v_1=w$ and $v_2=\frac{z}{b}\frac{dw}{dz}$ when $b\neq 0$.
Whence a solution $G(e_0,e_1)(z)$ of the KZ-equation yields 
two solutions $w_1$, $w_2$  with
$(\vec v(w_1), \vec v(w_2))=G(X_0,-Y_0)(z)$ 
of the above differential equation when it converges
as explained in \cite{O}.
%$G(X_0,-Y_0)(z)$ converges.
Specifically, we have
\begin{equation}\label{eq:HG=11}
 \HG{a,b}{c}{z}=[G_{\vec{01}}(X_0,-Y_0)(z)]_{11}
\end{equation}
where the right lower suffix $11$ means the $(1,1)$ entry.
By writing 
\begin{align}\label{eqC}
%\label{eqC:PhiG01}
&\V_{\vec{01}}(z):=
G_{\vec{01}}(X_0,-Y_0)(z)\cdot
\begin{pmatrix}
1 & 1 \\
0 & \frac{p}{b}
\end{pmatrix}
%\in \GL_2(\PP[\frac{1}{\bbb}])
, \\ 
\notag
%\label{eqC:PhiG10}
&\V_{\vec{10}}(z):=
G_{\vec{10}}(X_0,-Y_0)(z)\cdot
\begin{pmatrix}
1 & 0 \\
\frac{-a}{q} & \frac{q-1}{b}
\end{pmatrix}
%\in \GL_2(\PP[\frac{1}{\bbb\qqq}])
, \\ 
\notag
%\label{eqC:PhiG1infty}
&\V_{\vec{1\infty}}(z):=
G_{\vec{1\infty}}(X_0,-Y_0)(z)\cdot
\begin{pmatrix}
1 & 0 \\
\frac{-a}{q} & \Add{\frac{1-q}{b}}
\end{pmatrix}
%\in \GL_2(\PP[\frac{1}{\bbb\qqq}])
, \\ 
\notag
%\label{eqC:PhiGinfty1}
&\V_{\vec{\infty 1}}(z):=
G_{\vec{\infty 1}}(X_0,-Y_0)(z)\cdot
\begin{pmatrix}
1 & \Add{1} \\
\frac{-a}{b} & -1
\end{pmatrix}
%\in \GL_2(\PP[\frac{1}{\bbb}])
, \\ 
\notag
%\label{eqC:PhiGinfty0}
&\V_{\vec{\infty 0}}(z):=
G_{\vec{\infty 0}}(X_0,-Y_0)(z)\cdot
\begin{pmatrix}
1 & 1 \\
\frac{-a}{b} & -1
\end{pmatrix},\\
%\in \GL_2(\PP[\frac{1}{\bbb}])
%\label{eqC:PhiG0infty}
\notag
&\V_{\vec{0\infty}}(z):=
G_{\vec{0\infty}}(X_0,-Y_0)(z)\cdot
\begin{pmatrix}
1 & 1 \\
0 & \frac{p}{b}
\end{pmatrix}
%\in \GL_2(\PP[\frac{1}{\bbb}])
.
\end{align}
%with $\PP:=\Q_\ell[[\aaa,\bbb,\ccc-1]]$.
we recover certain scalar multiples of 
half of the so-called Kummer's 24 solutions
\begin{align}\label{eqD}
\begin{pmatrix} 1, & 0 \end{pmatrix}\cdot \V_{\vec{01}}(z) &=
\begin{pmatrix} \HG{a,b}{c}{z}, & z^{1-c}\HG{b+1-c,a+1-c}{2-c}{z}\end{pmatrix}, \\
\notag
\begin{pmatrix} 1, & 0 \end{pmatrix}\cdot \V_{\vec{10}}(z)
&=
\begin{pmatrix} \HG{a,b}{a+b+1-c}{1-z}, & (1-z)^{c-a-b}\HG{c-a,c-b}{c-a-b+1}{1-z}\end{pmatrix}, \\
\notag
\begin{pmatrix} 1, & 0 \end{pmatrix}\cdot \V_{\vec{1\infty}}(z)
&=
\begin{pmatrix} z^{-a}\HG{a,a+1-c}{a+b-c+1}{1-\frac{1}{z}}, & 
z^{b-c}(z-1)^{c-a-b}\HG{1-b,c-b}{1-a-b+c}{1-\frac{1}{z}}
\end{pmatrix}, \\
\notag
\begin{pmatrix} 1, & 0 \end{pmatrix}\cdot \V_{\vec{\infty 1}}(z)
&=
\begin{pmatrix} z^{-a}\HG{a,a+1-c}{a-b+1}{\frac{1}{z}}, & 
z^{-b}\HG{b+1-c,b}{b-a+1}{\frac{1}{z}}
\end{pmatrix}, \\
\notag
\begin{pmatrix} 1, & 0 \end{pmatrix}\cdot \V_{\vec{\infty 0}}(z)
&=
\begin{pmatrix} (1-z)^{-a}\HG{a,c-b}{a-b+1}{\frac{1}{1-z}},& 
(1-z)^{-b}\HG{c-a,b}{1-a+b}{\frac{1}{1-z}}
\end{pmatrix},\\
\notag
\begin{pmatrix} 1, & 0 \end{pmatrix}\cdot \V_{\vec{0\infty}}(z)
&=
\begin{pmatrix} (1-z)^{-a}\HG{a,c-b}{c}{\frac{z}{z-1}}, & 
(1-z)^{-a}(\frac{z}{z-1})^{1-c}\HG{1-b,a-c+1}{2-c}{\frac{z}{z-1}}
\end{pmatrix}
\end{align}
where we consider the branch  by the principal value
under  appropriate conditions for 
%and assume that 
$a,b,c$.
We note that the other half of the 24 solutions can be obtained by 
Euler's
%and Pfaff's 
transformation formula.

%%%%%%%%%%%%%%%%%%%%%%%%%%%%%%%%%%%%%%%%%%%%%%%%%%%%%%%%%%%%%%%%%%%%%%%
\section{$\ell$-adic analogues of the six solutions}
\label{sec:l-adic 6 solutions}
We recall Drinfeld's definition (\cite{Dr}) of associators and the 
Grothendieck-Teichm\"{u}ller group.
Then, we introduce $\ell$-adic analogues of the six solutions
of the KZ-equation and of the hypergeometric equation
discussed in \S\ref{sec:complex case}.
 
Let $\K$ be a field with characteristic $0$.
Let $\frak f_2$ be the free Lie algebra over $\K$ with two variables $e_0$ and $e_1$
and $U\frak f_2:=\K\langle e_0,e_1\rangle$ be its universal enveloping algebra.
We denote $\widehat{\frak f_2}$ and $\widehat{U\frak f_2}:=\K\langle\langle e_0,e_1\rangle\rangle$ to be their completions by
degrees. 
A {\it word} means a monic monomial element (including $1$) in $\widehat{U\frak f_2}$, and
for each $\varphi$ in $\widehat{U\frak f_2}$,
we denote $(\varphi|W)$ to be the coefficient of $\varphi$ in $W$.

\begin{defn}[\cite{Dr, F10}]\label{defn:associator}
(1).
The set  $M(\K)$ of {\it associators}
is a collection of pairs $(\mu,\varphi)$ with
$\mu\in\K^\times$ and $\varphi\in \widehat{U\mathfrak f_2}$
%A pair  $(\mu,\varphi)$ forms an {\it associator} if
 %A pair {\color{cyan}$(\mu,\varphi)$} 
%with {\color{cyan}$\mu\in\mathbb K^\times$} ($\mathbb K$: a field of ch$=0$) and
%a series $\varphi=\varphi(f_0,f_1)$ in
%\Q\langle\langle f_0,f_1\rangle\rangle=
%$\widehat{U\mathfrak f_2}$
%with non-zero quadratic terms
%which 
satisfying the following:
\begin{itemize}
\item the condition on the quadratic term: $(\varphi|e_0e_1)=\frac{\mu^2}{24}$,
\item the commutator group-like condition: $\varphi\in \exp [\hat{\mathfrak f}_2,\hat{\mathfrak f}_2]$,
%&\varphi(0,f_1)=\varphi(f_0,0)=1, \\
%&\varphi(1\otimes f_0 +f_0\otimes 1,1\otimes f_1 +f_1\otimes 1)=\varphi(f_0,f_1)\otimes \varphi(f_0,f_1), \\ %\  \varphi(0,0)=1
%and the {\color{magenta} pentagon equation} 
\item the pentagon equation:
$\varphi_{345}\varphi_{512}\varphi_{234}\varphi_{451}\varphi_{123}=1$
in  $\widehat{U\mathfrak P_5}$.
\end{itemize}
Here, $\exp [\hat{\mathfrak f}_2,\hat{\mathfrak f}_2]$ is the image of
the topological commutator $[\hat{\mathfrak f}_2,\hat{\mathfrak f}_2]$
of $\hat{\mathfrak f}_2$
under the exponential map, and
$\frak P_5$ is the Lie algebra generated by $t_{ij}$ ($i,j\in\Z/5$)
with the relations
\begin{itemize}
\item $t_{ij}=t_{ji}$, \quad $t_{ii}=0$,%\qquad  
\item $\sum\nolimits_{j\in\Z/5}t_{ij}=0\quad  (\forall i\in \Z/5)$,
\item $[t_{ij},t_{kl}]=0 \quad \text{for}\quad \{i,j\}\cap\{k,l\}=\emptyset$.
\end{itemize}
For  $i,j,k\in\Z/5$,  $\varphi_{ijk}$ means the image of $\varphi$ under the 
embedding $\widehat{U\mathfrak f_2}\to\widehat{U\mathfrak P_5}$
sending $e_0\mapsto t_{ij}$ and $e_1\mapsto t_{jk}$. 
A pair  $(\mu,\varphi)$ 
is called a  {\it rational} (resp. {\it $\ell$-adic) associator} when $\K$ is taken to be $\Q$ (resp. $\Q_\ell$ ).
A series $\varphi\in \widehat{U\mathfrak f_2}$ is called a {\it unitary associator}
when $(1,\varphi)$ forms an associator. 
A unitary associator is called {\it even} when $\varphi(-e_0,-e_1)=\varphi(e_0,e_1)$ holds.

(2). {\it The Grothendieck-Teichm\"{u}ller group} $\GT(\K)$ is the set of collections of pairs  $(\lambda, f)$
with $\lambda\in\K^\times$ and $f\in F_2(\K)$
% in the Malcev completion $F_2(\K)$ of $F_2$
satisfying the following:
\begin{itemize}
\item the condition on the quadratic term: $(f(e^{e_o},e^{e_1})|e_0e_1)=\frac{\lambda^2-1}{24}$
\item the commutator group-like condition: $f\in [F_2(\K),F_2(\K)]$
%&\varphi(0,f_1)=\varphi(f_0,0)=1, \\
%&\varphi(1\otimes f_0 +f_0\otimes 1,1\otimes f_1 +f_1\otimes 1)=\varphi(f_0,f_1)\otimes \varphi(f_0,f_1), \\ %\  \varphi(0,0)=1
%and the {\color{magenta} pentagon equation} 
\item the pentagon equation:
$
f(x^*_{12},x^*_{23})f(x^*_{34},x^*_{45})f(x^*_{51},x^*_{12})f(x^*_{23},x^*_{34})f(x^*_{45},x^*_{51})=1.
%\quad\text{ in } {P}^*_5(\mathbb K)
$
\end{itemize}
Here we denote the Malcev completion of $F_2$ by $F_2(\K)$ and
its topological commutator  by $[F_2(\K),F_2(\K)]$.
For any group homomorphism $F_2(\K)\to H$ sending $x_0\mapsto \alpha$
and $x_1\mapsto \beta$, the symbol $f(\alpha,\beta)$ stands for
the image of each $f\in F_2(\K)$. 
%We mean $f(e^{e_o},e^{e_1})$ to be the image of $f$ under the homomorphism
%$F_2(\K)\to \widehat{U\mathfrak f_2}^\times$.
The last equation is in the Malcev completion ${P}^*_5(\mathbb K)$
of the pure sphere braid group $P_5^*$ with $5$ strings,
and $x^*_{ij}$s are the standard generators (cf. \cite{EF1}).

(3).
In \cite{Dr}, it is shown that 
the above set-theoretically defined  ${\GT}(\mathbb K)$ indeed forms a %proalgebraic 
group
whose product is induced from that of $\mathrm{Aut}{F}_2(\mathbb K)$
and is given by
%\footnote{
%For our purpose to make \eqref{GT to Aut B_n} not anti-homomorphic but homomorphic, 
%we reverse the order of the product given in  the original paper \cite{Dr}.}
%\begin{equation*}\label{product of proalg-GT}
$$
(\lambda_1,f_1)\circledast (\lambda_2,f_2)
:=\Bigl(\lambda_2\lambda_1, f_1(f_2x^{\lambda_2}f_2^{-1},y^{\lambda_2})\cdot f_2\Bigr)
=\Bigl(\lambda_2\lambda_1,f_2\cdot f_1(x^{\lambda_2},f_2^{-1}y^{\lambda_2}f_2)\Bigr).
$$
%\end{equation*}
The associator set $M(\mathbb K)$ forms a left
$\GT(\mathbb K)$-torsor by
%the right $\GT(\mathbb K)$-action given by 
$$
 (\lambda,f)\circledast (\mu,\varphi):=\left(\lambda\mu, \ f(\varphi e^{\mu A}\varphi^{-1},e^{\mu B})\cdot\varphi\right)
=\left(\lambda\mu, \ \varphi\cdot f( e^{\mu A},\varphi^{-1}e^{\mu B}\varphi)\right)
$$
for $(\mu,\varphi)\in M(\mathbb K)$ and $(\lambda,f)\in \GT(\mathbb K)$.
\end{defn}

\begin{rem} 
In \cite{F10}, it is shown that the above definition of associators 
(resp. the Grothendieck-Teichm\"{u}ller group)
implies
the so-called {\it 2-cycle relation}
$$\varphi(e_0,e_1)\varphi(e_1,e_0)=1
\qquad(\text{ resp. }f(x_0,x_1)f(x_1,x_0)=1)
$$
and  the {\it 3-cycle relation}
\begin{align*}
&e^{\frac{\mu e_0}{2}}\varphi(e_\infty,e_0)e^{\frac{\mu e_\infty}{2}}\varphi(e_1,e_\infty)
e^{\frac{\mu e_1}{2}} \varphi(e_0,e_1)=1 \\
%with $e_0+e_1+e_\infty=0$.
(\text{resp.}&\qquad f(x_\infty,x_0)x_\infty^m f(x_1,x_\infty)x_1^m f(x_0,x_1)x^m=1
\quad\text{with}\quad m=\frac{\lambda-1}{2}),
\end{align*}
which are originally imposed  in the definition of
$M(\K)$ and $\GT(\K)$. 
\end{rem}

\begin{rem}[\cite{Dr}]
(1). %In \cite{Dr}, it is shown that 
Rational even unitary associators exist.

(2). %The KZ-associator $\Phi_{\mathrm KZ}$ in \eqref{eq:KZ associator}
The generating series $\varPhi_{\mathrm KZ}$ of multiple zeta values,
that is defined by 
$\varPhi_{\mathrm KZ}=G_{\vec{10}}(e_0,e_1)(z)^{-1}G_{\vec{01}}(e_0,e_1)(z)$ 
%constructed from the KZ-equation
forms an associator with $\mu=2\pi\sqrt{-1}$ and $\K=\C$. % (cf. \cite{Dr}). 

(3). The Galois action of $G_\Q$ on 
$\hat F_2^{\ell}=\hat\pi_1^\ell(\X_{\bar \Q};\vec{01})$
(in \S \ref{introduction})
induces a homomorphism
\begin{equation}\label{eq: GQ to GT}
G_\Q\to \GT(\Q_\ell). % (\cite{Dr}).
\end{equation}
\end{rem}

The associated  gamma function is one  of our main tools:

\begin{defn}[cf. \cite{EF1}]\label{defn: associated gamma function}
%Let  $(\mu,\varphi)$ be an associator.
%
%(1). The {\it fake comparison isomorphism}
%$$
%\comp_\varphi^{\vec{01}}: F_2(\K)\simeq F_2^{\mathrm{DR}}
%$$
%with $F_2^{\mathrm{DR}}:=\exp\hat {\mathfrak f_2}$ 
%is defined by $x\mapsto \exp(\mu e_0)$ and
%$y\mapsto \varphi^{-1}\exp(\mu e_1)\varphi$.
%
%(2). 
For an associator $(\mu,\varphi)$,  the {\it associated  gamma function}
is defined to be
$$\Gamma_{\varphi}(t):=\exp\left\{\sum_{n=1}^\infty 
\frac{(-1)^{n+1}}{n}
(\varphi|e_0^{n-1}e_1)t^n\right\}\in\K[[t]].$$
Similarly  for $(\lambda, f)\in \GT(\K)$,
%Similarly to  $\Gamma_\varphi$, 
it is defined to be
\begin{equation*}\label{eq:Gamma sigma}
\Gamma_{f}(t):=\exp\left\{\sum_{n=1}^\infty 
\frac{(-1)^{n+1}}{n}
\left(f(e^{e_0},e^{e_1})|e_0^{n-1}e_1\right)t^n\right\}
\in\K[[t]].
\end{equation*}
\end{defn}

\begin{rem}\label{rem:Gamma}
(1).
%We note that 
When $(\mu,\varphi)=(2\pi\sqrt{-1},\varPhi_{\mathrm KZ})$, it is equal to
$e^{\gamma t}\Gamma (1+t)=\exp{\{\sum_{n=2}^\infty\frac{\zeta(n)}{n}(-t)^n}\}$
where $\Gamma$ is the classical gamma function  and
$\gamma$ is the Euler-Mascheroni constant.

(2). When $\varphi$ is an even unitary associator, 
we have $\Gamma_\varphi(t)=\Gamma_+(t)$ (cf. \cite{EF1} Lemma 9.5).

(3). Let $\sigma\in G_\Q$, which corresponds to 
$(\lambda_\sigma, f_\sigma)\in \GT(\Q_\ell)$
under \eqref{eq: GQ to GT}.
%Similarly to  $\Gamma_\varphi$, 
Write $\Gamma_{\sigma}(t):=\Gamma_{f_\sigma}(t)$.
Then, it is calculated to be
$$
\Gamma_{\sigma}(t)=\exp\left\{
\sum_{m>1}%\limits_{m\geqslant 3:{\text{odd}}}
\frac{\kappa^{(l)*}_m(\sigma)}{m!}t^m
\right\}
$$
($\kappa^{(l)*}_m(\sigma)$ %\in{\mathcal O}(\underline{Gal}^{(l)}_{\bold Q})$
is %calculated to be 
the $\ell$-adic $m$-th Soul\'{e} cocycle %character
when $m$ is odd),
which is the series related to Anderson's (\cite{A}) hyperadelic gamma function 
(consult \cite{Ih90} and \cite{F06}).

(4).
When $ (\mu',\varphi')=(\lambda,f)\circledast (\mu,\varphi)$,
we have
\begin{equation}\label{eq: Gamma composition}
\Gamma_{\varphi'}(t)=\Gamma_f(\mu t)\Gamma_\varphi(t)
\end{equation}
by definition.
\end{rem}

Isomorphisms between 
$F_2(\K)$ and $\exp\hat {\mathfrak f_2}$
deduced by associators are discussed in
the literature (cf. \cite{BN, EF1}):

\begin{defn}%[cf. \cite{BN, EF1}]
\label{defn: fake comparison}
For an associator  $(\mu,\varphi)$, %be an associator.
the {\it fake comparison isomorphism}
$$
\comp_\varphi^{\vec{01}}: F_2(\K)\to %F_2^{\mathrm{DR}}(\K)
\exp\hat {\mathfrak f_2}
$$
is the isomorphism given by $x\mapsto \exp(\mu e_0)$ and
$y\mapsto \varphi^{-1}\exp(\mu e_1)\varphi$.
%with $F_2^{\mathrm{DR}}(\K):=\exp\hat {\mathfrak f_2}$.
%
%(2). For an associator $(\mu,\varphi)$,  the {\it associated  gamma function}
%is defined to be
%$$\Gamma_{\varphi}(t):=\exp\left\{\sum_{n=1}^\infty 
%\frac{(-1)^{n+1}}{n}
%(\varphi|e_0^{n-1}e_1)t^n\right\}\in\K[[t]].$$
%Similarly  for $(\lambda, f)\in \GT(\K)$,
%%Similarly to  $\Gamma_\varphi$, 
%it is defined to be
%\begin{equation*}\label{eq:Gamma sigma}
%\Gamma_{f}(t):=\exp\left\{\sum_{n=1}^\infty 
%\frac{(-1)^{n+1}}{n}
%\left(f(e^{e_0},e^{e_1})|e_0^{n-1}e_1\right)t^n\right\}
%\in\K[[t]].
%\end{equation*}
\end{defn}

\begin{rem}
When $(\mu,\varphi)=(2\pi\sqrt{-1},\varPhi_{\mathrm KZ})$,
it agrees with the Betti-de Rham comparison isomorphism of 
the motivic fundamental group
$\pi_1^{\mathcal M}({\mathcal X},\vec{01}) $ (cf. \cite{De}).
\end{rem}

Under the natural inclusion $F_2^\ell\hookrightarrow F_2(\Q_\ell)$,
we regard $F_2^\ell$ as a topological subgroup of $F_2(\Q_\ell)$ (equipped with the $\ell$-adic topology).
For an $\ell$-adic  unitary associator $\varphi$,
we define 
\begin{equation}\label{eq:iota-varphi}
\iota_{\varphi}:=\comp_\varphi^{\vec{01}}|_{F_2^\ell},
%:\widehat F_2^{\ell}\to F_2^{\mathrm{DR}}(\Q_\ell)
\end{equation}
that is,
the continuous group homomorphism
\begin{equation*}%\label{eq:iota-varphi}
\iota_{\varphi}:\widehat F_2^{\ell}\to F_2^{\mathrm{DR}}(\Q_\ell)
\quad (\subset\Q_\ell\langle\langle e_0,e_1 \rangle\rangle^\times)
\end{equation*}
sending $x_0\mapsto e^{e_0}$ and
$x_1\mapsto \varphi^{-1}e^{e_1}\varphi$.
Then, we have
$\iota_\varphi(x_\infty)
=\Ad\left(\varphi(e_0,e_\infty)e^{-\frac{e_0}{2}}\right)^{-1}(e^{e_\infty})$
by the 2- and 3-cycle relations for $\varphi$.
Here, for invertible elements $u$ and $v$, we mean
$\Ad(u)(v)=uvu^{-1}$.  
%For an $\ell$-adic even unitary associator $\varphi$, % be an even associator.
%We consider the homomorphism
%$\iota_{(\mu,\varphi)}:\widehat F_2^{(\ell)}\to \Q_\ell\langle\langle e_0,e_1 \rangle\rangle$
%sending $x_0\mapsto \exp(\mu e_0)$ and
%$x_1\mapsto \varphi^{-1}\exp(\mu e_1)\varphi$.
We write
$$
G_{\vec{01}}^{\varphi}(e_0,e_1)(\sigma)(z):=\iota_{\varphi}(f_\sigma^z)
\in \Q_\ell\langle\langle e_0,e_1 \rangle\rangle.
$$
Following \eqref{eqB} in the complex case, %(\S\ref{sec:complex case}), 
we consider
\begin{align*}
G_{\vec{10}}^{\varphi}(e_0,e_1)(\sigma)(z)
&:=G_{\vec{01}}^{\varphi}(e_1,e_0)(\sigma)(1-z), \\
G_{\vec{1\infty}}^{\varphi}(e_0,e_1)(\sigma)(z)
&:=%\Ad\left(e^{\frac{e_1}{2}}\varphi(e_0,e_1) \right)^{-1}\left( 
G_{\vec{01}}^{\varphi}(e_1,e_\infty)(\sigma)(1-\frac{1}{z})
%\right)
, \\
G_{\vec{\infty 1}}^{\varphi}(e_0,e_1)(\sigma)(z)
&:=%\Ad\left(\varphi(e_1,e_\infty)e^{\frac{e_1}{2}}\varphi(e_0,e_1) \right)^{-1}\left( 
G_{\vec{01}}^{\varphi}(e_\infty,e_1)(\sigma)(\frac{1}{z})%\right)
, \\
G_{\vec{\infty 0}}^{\varphi}(e_0,e_1)(\sigma)(z)
&:=%\Ad\left(e^{\frac{e_\infty}{2}}\varphi(e_1,e_\infty)e^{\frac{e_1}{2}}\varphi(e_0,e_1) \right)^{-1}\left( 
G_{\vec{01}}^{\varphi}(e_\infty,e_0)(\sigma)(\frac{1}{1-z})
%\right)
, \\
G_{\vec{0\infty}}^{\varphi}(e_0,e_1)(\sigma)(z)
&:=%\Ad\left(e^{\frac{e_0}{2}}\right)\left(
G_{\vec{01}}^{\varphi}(e_0,e_\infty)(\sigma)(\frac{z}{z-1})%\right)
\end{align*}
which we regard as $\ell$-adic analogues of the six fundamental solutions of the
KZ-equation. 

\begin{rem}
We note  that the $\ell$-adic polylogarithms discussed in \cite{NW, W}
can be extracted as coefficients of coefficients of
$G_{\vec{01}}^{\varphi}(e_0,e_1)(\sigma)(z)$.
\end{rem}

Let $\aaa,\bbb,\ccc$ be variables.
Write $\ppp=1-\ccc$, $\qqq=\aaa+\bbb+1-\ccc=\aaa+\bbb+\ppp$ and
$$
X=
\begin{pmatrix}
0 & \bbb \\
0 & \ppp
\end{pmatrix}, \
Y=
\begin{pmatrix}
0 & 0 \\
\aaa & \qqq
\end{pmatrix}\in \Mat_2(\K[\aaa,\bbb,\ccc-1]).
$$
Write $\PP:=\Q_\ell[[\aaa,\bbb,\ccc-1]]$. 
Let $\varphi$ be an (even) unitary associator.
Following \eqref{eqC}, 
% \ref{sec:complex case},
we consider $\ell$-adic analogues of six solutions %(\S \ref{sec:complex case})
of the hypergeometric equation:
\begin{align}\label{eq:PhiG01}
&\V_{\vec{01}}^{\varphi}(\sigma)(z):=
G_{\vec{01}}^{\varphi}(X,-Y)(\sigma)(z)\cdot
\begin{pmatrix}
1 & 1 \\
0 & \frac{\ppp}{\bbb}
\end{pmatrix}
\in \GL_2(\PP[\frac{1}{\bbb}]), \\ 
\label{eq:PhiG10}
&\V_{\vec{10}}^{\varphi}(\sigma)(z):=
G_{\vec{10}}^{\varphi}(X,-Y)(\sigma)(z)\cdot
\begin{pmatrix}
1 & 0 \\
\frac{-\aaa}{\qqq} & \frac{\qqq-1}{\bbb}
\end{pmatrix}
\in \GL_2(\PP[\frac{1}{\bbb\qqq}]), \\
\label{eq:PhiG1infty}
&\V_{\vec{1\infty}}^{\varphi}(\sigma)(z):=
G_{\vec{1\infty}}^{\varphi}(X,-Y)(\sigma)(z)\cdot
\begin{pmatrix}
1 & 0 \\
\frac{-\aaa}{\qqq} & \Add{\frac{1-\qqq}{\bbb}}
\end{pmatrix}
\in \GL_2(\PP[\frac{1}{\bbb\qqq}]), \\
\label{eq:PhiGinfty1}
&\V_{\vec{\infty 1}}^{\varphi}(\sigma)(z):=
G_{\vec{\infty 1}}^{\varphi}(X,-Y)(\sigma)(z)\cdot
\begin{pmatrix}
1 & \Add{1} \\
\frac{-\aaa}{\bbb} & -1
\end{pmatrix}
\in \GL_2(\PP[\frac{1}{\bbb}]), \\ 
\label{eq:PhiGinfty0}
&\V_{\vec{\infty 0}}^{\varphi}(\sigma)(z):=
G_{\vec{\infty 0}}^{\varphi}(X,-Y)(\sigma)(z)\cdot
\begin{pmatrix}
1 & 1 \\
\frac{-\aaa}{\bbb} & -1
\end{pmatrix}
\in \GL_2(\PP[\frac{1}{\bbb}]), \\
\label{eq:PhiG0infty}
&\V_{\vec{0\infty}}^{\varphi}(\sigma)(z):=
G_{\vec{0\infty}}^{\varphi}(X,-Y)(\sigma)(z)\cdot
\begin{pmatrix}
1 & 1 \\
0 & \frac{\ppp}{\bbb}
\end{pmatrix}
\in \GL_2(\PP[\frac{1}{\bbb}]).
\end{align}
These matrices will be employed to relax our assumption \eqref{eq:ABC condition}
and to show
%Pfaff's and 
Euler's transformation formula
for the $\ell$-adic hypergeometric function
in \S \ref{sec:Euler's transformation formula}.

%%%%%%%%%%%%%%%%%%%%%%%%%%%%%%%%%%%%%%%%%%%%%%%%%%%%%%%%%%%%%%%%%%%%%%%
\section{$\ell$-adic Gauss hypergeometric theorem and Ohno-Zagier relation}
\label{sec:Gauss's hypergeometric theorem}
In this section,
we present a proof of
the fundamental theorem 
(Theorem \ref{fundamental theorem of l-adic HG}) and
the $\ell$-adic Gauss hypergeometric theorem
(Theorem \ref{thm: l-adic Gauss's hypergeometric theorem}). 
%as well as the key theorem (Theorem \ref{thm:even unitary associator=M+0}) 
To do so, we present  a key result (Theorem \ref{thm: varphi=M})
that the evaluation of the two-by-two matrices
$X$ and $Y$ for any  associator is described explicitly in terms of
the associated gamma function,
where the Ohno-Zagier relation is observed as the $(1,1)$ entry,
which implies Theorem \ref{thm:even unitary associator=M+0}.

\begin{nota}\label{notation for M}
For an associator  $(\mu,\varphi)\in \K^\times\times\widehat{U\mathfrak f_2}$,
we write 
$$
\GGamma{\varphi}{s,t}{u,v}:=
\frac{\Gamma_{\varphi}(s)\Gamma_{\varphi}(t)}{\Gamma_{\varphi}(u)\Gamma_{\varphi}(v)}
\in\K[[s,t,u,v]]
$$
where $\Gamma_{\varphi}$ is the associated gamma function  (cf. Definition \ref{defn: associated gamma function}).
We consider 
the associated $2\times 2$ matrix 
$$
M_{\varphi}
:=
\begin{pmatrix}
1 & 0 \\
-\frac{\aaa}{\qqq} & \frac{\ppp}{\bbb}
\end{pmatrix}
C_{\varphi}
\begin{pmatrix}
1 & -\frac{\bbb}{\ppp} \\
0 & \frac{\bbb}{\ppp}
\end{pmatrix}  \ \ 
\in \GL_2(\K[[\aaa,\bbb,\ccc-1]][\ppp^{-1},\qqq^{-1}])
$$ 
with
$\ppp=1-\ccc$, $\qqq=\aaa+\bbb+1-\ccc=\aaa+\bbb+\ppp$ and
$$
C_{\varphi}:=
\begin{pmatrix}
\GGamma{\varphi}{-\ppp,-\qqq}{-\ppp-\aaa,-\ppp-\bbb}
& 
\GGamma{\varphi}{\ppp,-\qqq}{-\aaa,-\bbb}
\\
\frac{\aaa\bbb}{\ppp\qqq}\GGamma{\varphi}{-\ppp,\qqq}{\aaa,\bbb}
&
\frac{(\aaa+\ppp)(\bbb+\ppp)}{\ppp\qqq}\GGamma{\varphi}{\ppp,\qqq}{\ppp+\aaa,\ppp+\bbb}
\end{pmatrix}.
$$
\end{nota}

\begin{prop}\label{prop: M in SL2}
%For an associator  $(\mu,\varphi)\in \K^\times\times\widehat{U\mathfrak f_2}$,
We have
$M_\varphi\in \SL_2(\K[[\aaa,\bbb,\ccc-1]])$.
\end{prop}

\begin{proof}
The matrix $M_\varphi$ is calculated to be
{%\small
\begin{equation*}%\label{eq:M+0}
M_{\varphi}=
\begin{pmatrix}
\GGamma{\varphi}{-\ppp,-\qqq}{-\ppp-\aaa,-\ppp-\bbb}
&
\frac{\bbb}{\ppp}\left\{\GGamma{\varphi}{\ppp,-\qqq}{-\aaa,-\bbb}-\GGamma{\varphi}{-\ppp,-\qqq}{-\ppp-\aaa,-\ppp-\bbb}
\right\} \\
\frac{\aaa}{\qqq}\left\{\GGamma{\varphi}{-\ppp,\qqq}{\aaa,\bbb}-\GGamma{\varphi}{-\ppp,-\qqq}{-\ppp-\aaa,-\ppp-\bbb}\right\}
& [M_{\varphi}]_{22}
%\frac{(a+p)(b+p)}{pq}\GGamma{+}{p,q}{p+a,p+b}-\frac{ab}{pq}
%\left\{\GGamma{+}{-p,-q}{-p-a,-p-b}+\GGamma{+}{-p,q}{a,b}
%-\GGamma{+}{-p,-q}{-p-a,-p-b}
%\right\}
\end{pmatrix}
\end{equation*}
}
%in $\Mat_2(\K[[\aaa,\bbb,\ccc-1]])$
with 
%{\footnotesize
\begin{align*}
[M_{\varphi}]_{22}=&
\frac{(\aaa+\ppp)(\bbb+\ppp)}{\ppp\qqq}\GGamma{\varphi}{\ppp,\qqq}{\ppp+\aaa,\ppp+\bbb} \\
&\qquad +\frac{\aaa\bbb}{\ppp\qqq}
\left\{\GGamma{\varphi}{-\ppp,-\qqq}{-\ppp-\aaa,-\ppp-\bbb}-\GGamma{\varphi}{-\ppp,\qqq}{\aaa,\bbb}
-\GGamma{\varphi}{\ppp,-\qqq}{-\aaa,-\bbb}
\right\}.
\end{align*}

It is %immediate to see that 
clear that
the $(1,1)$ entry $[M_\varphi]_{11}$,
the $(1,2)$ entry $[M_\varphi]_{12}$ and
the $(2,1)$ entry $[M_\varphi]_{21}$  are in 
$\K[[\aaa,\bbb,\ccc-1]]$.
Thanks to the identity
$$
\Gamma_\varphi(t)\Gamma_\varphi(-t)=\frac{\mu t}{e^{\frac{\mu t}{2}}-e^{\frac{-\mu t}{2}}}=\frac{{\mu t}/2}{\sinh{{\mu t}/2}}.
$$
shown in \cite{EF3} Remark 4.6, we have
\begin{align*}
&\det M_\varphi=\det C_\varphi \\
\ &=\frac{\aaa\bbb+\ppp\qqq}{\ppp\qqq}\frac{\Gamma_\varphi(\ppp)\Gamma_\varphi(-\ppp)\Gamma_\varphi(\qqq)\Gamma_\varphi(\qqq)}{\Gamma_\varphi(\ppp+\aaa)\Gamma_\varphi(-\ppp-\aaa)\Gamma_\varphi(\ppp+\bbb)\Gamma_\varphi(-\ppp-\bbb)}
-\frac{\aaa\bbb}{\ppp\qqq}\frac{\Gamma_\varphi(\ppp)\Gamma_\varphi(-\ppp)\Gamma_\varphi(\qqq)\Gamma_\varphi(\qqq)}{\Gamma_\varphi(\aaa)\Gamma_\varphi(-\aaa)\Gamma_\varphi(\bbb)\Gamma_\varphi(-\bbb)} \\
&=\frac{\sinh(\frac{\mu}{2}(\ppp+\aaa))\sinh(\frac{\mu}{2}(\ppp+\bbb))}{\sinh(\frac{\mu}{2}\ppp)\sinh(\frac{\mu}{2}\qqq)}
-\frac{\sinh(\frac{\mu}{2}\aaa)\sinh(\frac{\mu}{2}\bbb)}{\sinh(\frac{\mu}{2}\ppp)\sinh(\frac{\mu}{2}\qqq)}\\
&=\frac{1}{\sinh(\frac{\mu}{2}\ppp)\sinh(\frac{\mu}{2}\qqq)}\cdot
\frac{e^\frac{\mu\ppp}{2}-e^{-\frac{\mu\ppp}{2}}}{2}\cdot
\frac{e^\frac{\mu\qqq}{2}-e^{-\frac{\mu\qqq}{2}}}{2}
=1.
\end{align*}
Thus, the $(2,2)$ entry $[M_\varphi]_{22}$  must also be in  $\K[[\aaa,\bbb,\ccc-1]]$.
Therefore, we have
$M_\varphi\in \SL_2(\K[[\aaa,\bbb,\ccc-1]])$.
\end{proof}

%Let $\aaa,\bbb,\ccc$ be variables.
%Put $\ppp=1-\ccc$, $\qqq=\aaa+\bbb+1-\ccc=\aaa+\bbb+\ppp$ and
%$$
%X=
%\begin{pmatrix}
%0 & \bbb \\
%0 & \ppp
%\end{pmatrix}, \
%Y=
%\begin{pmatrix}
%0 & 0 \\
%\aaa & \qqq
%\end{pmatrix}\in \Mat_2(\K[\aaa,\bbb,\ccc-1]).
%$$
We denote by $H(X,-Y)$
the image of each element $H\in \K\langle\langle e_0,e_1 \rangle\rangle$
under the map
\begin{equation}\label{eq: ev(X,-Y)}
\ev_{(X,-Y)}: 
\K\langle\langle e_0,e_1 \rangle\rangle
\to \Mat_2(
\K[[\aaa,\bbb,\ccc-1]])
\end{equation}
sending $e_0$ and $e_1$ to $X$ and $-Y$ respectively.

\begin{thm}\label{thm: varphi=M}
For an associator  $(\mu,\varphi)\in \K^\times\times\widehat{U\mathfrak f_2}$, 
we have
$$
\varphi(X,-Y)=M_{\varphi}.
$$
Particularly the $(1,1)$ entry  proves the relation of Ohno-Zagier in \cite{OZ}
for associators.
\end{thm}

\begin{proof}
We show the equation entry-wise.

{\it The $(1,1)$ entry:}
The main result of \cite{L} states that if the system 
$\{z(k_1,\dots,k_m)\in\C \bigm| m, k_1,\dots, k_{m-1}\geqslant 1, k_m>1  \}$ satisfies the regularized double shuffle relations  (cf. loc.~cit.),
the following Ohno-Zagier relation \cite{OZ} holds
\begin{align}\label{eq:Li formula}
1+ \aaa\bbb\sum_{\substack{k,n,s>0 \\k>n+s, \ n \geqslant s}}
&g_0(k,n,s)\ppp^{k-n-s}\qqq^{n-s}
(\aaa\bbb+\ppp\qqq)^{s-1}  \\
&=\exp[\sum_{n=2}^\infty\frac{z(n)}{n}\{\ppp^n+\qqq^n-(\aaa+\ppp)^n-(\bbb+\ppp)^n\}],  \label{eq:Li formula 2}
%\notag
\end{align}
where
\begin{equation}\label{eq: g0}
g_0(k,n,s)=
\sum_{\substack{\wt({\mathbf k})=k, \ \mathrm{dp}({\mathbf k})=n, \ \mathrm{ht}({\mathbf k})=s \\ {\mathbf k}:\text{ admissible index}}} z({\mathbf k}).
\end{equation}
Here, an admissible index means a tuple ${\mathbf k}=(k_1,\dots,k_n)\in\N^n$ ($n\in\N$) 
with $k_n>1$, and we write
$\wt({\mathbf k})=k_1+\cdots+k_n$, 
$\mathrm{dp}({\mathbf k})=n$, 
$\mathrm{ht}({\mathbf k})=\sharp\{i\bigm| k_i>1\}$. 
In \cite{F11}, is is shown that the coefficients of any associator $\varphi$
with an appropriate signature, 
$\zeta_\varphi(k_1,\dots,k_m)=(-1)^m
(\varphi|e_0^{k_m-1}e_1\cdots e_0^{k_1-1}e_1)$ in precise,
satisfy the regularized double shuffle relations.
Since %the left hand side of 
\eqref{eq:Li formula} for $\zeta_\varphi(k_1,\dots,k_m)$
is nothing but the $(1,1)$ entry $[\varphi(X,-Y)]_{11}$
by \cite{O} Lemma 3.1 and (65),
we obtain the equality 
$$
[\varphi(X,-Y)]_{11}=\GGamma{\varphi}{-\ppp,-\qqq}{-\ppp-\aaa,-\ppp-\bbb}
=[M_\varphi]_{11}.
$$

{\it The $(1,2)$ entry:}
By \cite{O} (66) and its following remark, we have
\begin{align*}
[\varphi(X,-Y)]_{11}+&\frac{\ppp}{\bbb}[\varphi(X,-Y)]_{12} \\
=
1+&(\aaa\bbb+\ppp\qqq)
\sum_{\substack{k,n,s>0 \\k>n+s, \ n \geqslant s}}
g_{\varphi}(k,n,s|{e_0{U\frak f_2}e_1})(-\ppp)^{k-n-s}\qqq^{n-s}
(\aaa\bbb)^{s-1} 
\end{align*}
where  $g_{\varphi}(k,n,s|{e_0{U\frak f_2}e_1})$ is defined analogously to 
\eqref{eq: g0} with $ z({\mathbf k})=\zeta_\varphi({\mathbf k})$.
Since the right-hand side of the above is obtained from 
\eqref{eq:Li formula 2} by the change of variables
$\aaa\mapsto \aaa+\ppp$,
$\bbb\mapsto \bbb+\ppp$,
$\ppp\mapsto -\ppp$,
$\qqq\mapsto\qqq$,
we have
$$
[\varphi(X,-Y)]_{11}+\frac{\ppp}{\bbb}[\varphi(X,-Y)]_{12}=
\GGamma{\varphi}{\ppp,-\qqq}{-\aaa,-\bbb}.
$$
Hence, we obtain the equality
\begin{equation}\label{eq:12}
[\varphi(X,-Y)]_{12}
=\frac{\bbb}{\ppp}\left\{\GGamma{\varphi}{\ppp,-\qqq}{-\aaa,-\bbb}
-\GGamma{\varphi}{-\ppp,-\qqq}{-\ppp-\aaa,-\ppp-\bbb}
\right\}
=[M_\varphi]_{12}.
\end{equation}

{\it The $(2,1)$ entry:}
By \cite{O} Lemma 3.1, we have
\begin{align}\label{eq:12aux}
[\varphi(X,-Y)]_{12}=&
\bbb\qqq\sum_{\substack{k,n,s>0 \\k>n+s, \ n \geqslant s}}
g_\varphi(k,n,s|{e_0{U\frak f_2}})\ppp^{k-n-s}\qqq^{n-s}
(\aaa\bbb+\ppp\qqq)^{s-1},  \\ \notag
[\varphi(X,-Y)]_{21}=&
\aaa\ppp\sum_{\substack{k,n,s>0 \\k>n+s, \ n \geqslant s}}
g_\varphi(k,n,s|{{U\frak f_2}e_1})\ppp^{k-n-s}\qqq^{n-s}
(\aaa\bbb+\ppp\qqq)^{s-1},  
\end{align}
where
\begin{align*}
g_\varphi(k,n,s|{e_0{U\frak f_2}})=&
\sum_{\substack{\wt(W)=k, \ \mathrm{dp}(W)=n, \ \mathrm{ht}(W)=s \\ W\in e_0{U\frak f_2}: \text{ word}}}
(-1)^{\mathrm{dp}(W)}(\varphi|W), \\
g_\varphi(k,n,s|{{U\frak f_2}e_1})=&
\sum_{\substack{\wt(W)=k, \ \mathrm{dp}(W)=n, \ \mathrm{ht}(W)=s \\ W\in {U\frak f_2}e_1: \text{ word}}}
(-1)^{\mathrm{dp}(W)}(\varphi|W).
\end{align*}
Here, for each word $W$,
$\wt(W)$ and $\mathrm{dp}(W)$ 
are defined to be the number of letters in $W$ and
the number of $e_1$ appearing in $W$, respectively.
And $\mathrm{ht}(W)-1$ is defined to be
the number of $e_1e_0$ appearing in $W$ (see \cite{O} Lemma 3.1).
By the 2-cycle relation and group-like condition for $\varphi$,
we have
$$
(-1)^{\mathrm{dp}(W)}(\varphi|W)=
(-1)^{\mathrm{dp}(W^\ast)}(\varphi|W^\ast)
$$
(for example, see \cite{F20+} Lemma 3.2).
Here, $W^\ast$ is the dual word of $W$, that is, the image of $W$ under the
anti-automorphism of  $U\frak f_2$ which switches $e_0$ and $e_1$.
Thus, we have
$$
g_\varphi(k,n,s|{{U\frak f_2}e_1})=
g_\varphi(k,k-n,s|{e_0{U\frak f_2}}).
$$
Hence,
\begin{align*}
[\varphi(X,-Y)]_{21}=&
\aaa\ppp\sum_{\substack{k,n,s>0 \\k>n+s, \ n \geqslant s}}
g_\varphi(k,k-n,s|{e_0{U\frak f_2}})\ppp^{k-n-s}\qqq^{n-s}
(\aaa\bbb+\ppp\qqq)^{s-1}   \\
=&
\aaa\ppp\sum_{\substack{k,m,s>0 \\k>m+s, \ m \geqslant s}}
g_\varphi(k,m,s|{e_0{U\frak f_2}})\ppp^{m-s}\qqq^{k-m-s}
(\aaa\bbb+\ppp\qqq)^{s-1}.
\intertext{By \eqref{eq:12} and \eqref{eq:12aux}, the change of variables 
$\aaa\mapsto -\aaa$, $\bbb\mapsto -\bbb$, $\ppp \mapsto \qqq$, $\qqq\mapsto \ppp$, we have		}
=&
\frac{\aaa}{\qqq}\left\{\GGamma{\varphi}{-\ppp,\qqq}{\aaa,\bbb}
-\GGamma{\varphi}{-\ppp,-\qqq}{-\ppp-\aaa,-\ppp-\bbb}
\right\}
=[M_\varphi]_{21}.
\end{align*}

{\it The $(2,2)$ entry:}
%By the decomposition
%$
%M_{\varphi}
%=
%\begin{pmatrix}
%1 & 0 \\
%-\frac{\aaa}{\qqq} & \frac{\ppp}{\bbb}
%\end{pmatrix}
%C_{\varphi}
%\begin{pmatrix}
%1 & -\frac{\bbb}{\ppp} \\
%0 & \frac{\bbb}{\ppp}
%\end{pmatrix}
%$ 
%%with
%%\in \GL_2(\Q_\ell[[\aaa,\bbb,\ccc-1]])
%%$$ 
%%$$
%%C_{\varphi}:=
%%\begin{pmatrix}
%%\GGamma{\varphi}{-\ppp,-\qqq}{-\ppp-\aaa,-\ppp-\bbb}
%%& 
%%\GGamma{\varphi}{\ppp,-\qqq}{-\aaa,-\bbb}
%%\\
%%\frac{\aaa\bbb}{\ppp\qqq}\GGamma{\varphi}{-\ppp,\qqq}{\aaa,\bbb}
%%&
%%\frac{(\aaa+\ppp)(\bbb+\ppp)}{\ppp\qqq}\GGamma{\varphi}{\ppp,\qqq}{\ppp+\aaa,\ppp+\bbb}
%%\end{pmatrix}.
%%$$
%%\quad 
%%$$
%%\GGamma{\varphi}{s,t}{u,v}:=
%%\frac{\Gamma_{\varphi}(s)\Gamma_{\varphi}(t)}{\Gamma_{\varphi}(u)\Gamma_{\varphi}(v)}
%%\quad \text{and} \quad
%%\Gamma_{\varphi}(t):=\exp\left\{\sum_{n=1}^\infty 
%%\frac{(-1)^{n+1}}{n}
%%(\varphi|e_0^{n-1}e_1)t^n\right\}
%%$$
%and the equation
%$\Gamma_\varphi(t)\Gamma_\varphi(-t)=\frac{\mu t}{e^{\frac{\mu t}{2}}-e^{\frac{-\mu t}{2}}}$
%in \cite{EF3} Remark 4.6, 
By Proposition \ref{prop: M in SL2},
%we obtain
$\det(M_\varphi)=1$.
While we have
$\det(\varphi(X,-Y))=1$
because $\varphi$  is commutator group-like.
These two claims assert that
$[\varphi(X,-Y)]_{22}=[M_\varphi]_{22}$.
\end{proof}

%Since $\varphi$ lies on the commutator of the group of group-like series,
%we see
%\begin{equation}
%M_\varphi\in \SL_2(\Q_\ell[[\aaa,\bbb,\ccc-1]]).
%\end{equation}

\begin{nota}\label{nota: M+}
We define the matrix
$$M_+\in\Mat_2(\K [[\aaa,\bbb,\ccc-1]])$$
in the same way as %$M_{\varphi}$.
%It is also same as 
$M_\varphi$ in Notation \ref{notation for M}
whose $\Gamma_\varphi(t)$ is replaced with
\begin{align*}
\Gamma_+(t)&=\sqrt{\frac{t}{2}\mathrm{cosech}(\frac{t}{2})}
=\sqrt{\Gamma(1+\frac{t}{2\pi \sqrt{-1}})\Gamma(1-\frac{t}{2\pi \sqrt{-1}})}\\
&=\sqrt{\frac{t}{e^\frac{t}{2}-e^{-\frac{t}{2}}}}
=\exp\left\{-\sum_{n=1}^\infty \frac{B_{2n}}{2(2n)!}t^{2n}\right\}.
\end{align*}
Recall that the symbol $\Gamma$ means the classical gamma function in Remark \ref{rem:Gamma}. (1).
Namely, $M_+$ is given by 
\begin{equation*}\label{eq:M+0}
%M_{+}=
\begin{pmatrix}
1 & 0 \\
-\frac{\aaa}{\qqq} & \frac{\ppp}{\bbb}
\end{pmatrix}
C_{+}
\begin{pmatrix}
1 & -\frac{\bbb}{\ppp} \\
0 & \frac{\bbb}{\ppp}
\end{pmatrix}% \in \GL_2(\Q_\ell)
%\end{equation*}
\text{ with }
%$$
C_{+}:=
\begin{pmatrix}
%\frac{\Gamma_+(p)\Gamma_+(q)}{\Gamma_+(p+a)\Gamma_+(p+b)} 
\GGamma{+}{-\ppp,-\qqq}{-\ppp-\aaa,-\ppp-\bbb}
& 
%\frac{\Gamma_+(p)\Gamma_+(q)}{\Gamma_+(a)\Gamma_+(b)} 
\GGamma{+}{\ppp,-\qqq}{-\aaa,-\bbb}
\\
%\frac{ab}{q(q-1)}\frac{\Gamma_+(p)\Gamma_+(q)}{\Gamma_+(a)\Gamma_+(b)} 
\frac{\aaa\bbb}{\ppp\qqq}\GGamma{+}{-\ppp,\qqq}{\aaa,\bbb}
&
%\frac{(a+p)(b+p)}{q(q-1)}\frac{\Gamma_+(p)\Gamma_+(q)}{\Gamma_+(p+a)\Gamma_+(p+b)}
\frac{(\aaa+\ppp)(\bbb+\ppp)}{\ppp\qqq}\GGamma{+}{\ppp,\qqq}{\ppp+\aaa,\ppp+\bbb}
\end{pmatrix}
%\quad \GGamma{+}{s,t}{u,v}=\frac{\Gamma_+(s)\Gamma_+(t)}{\Gamma_+(u)\Gamma_+(v)}
\end{equation*}
and
$\GGamma{+}{s,t}{u,v}:=\frac{\Gamma_+(s)\Gamma_+(t)}{\Gamma_+(u)\Gamma_+(v)}$.
\end{nota}

%Theorem \ref{thm:even unitary associator=M+0} is a consequence of the above theorem.

%For $\Gamma_{\varphi}$ see also \cite{EF}.
%\begin{cor}\label{cor: varphi=M+}
%In particular if $\varphi$ is an even unitary associator, % with $\mu=1$,
%we have
%$$\varphi(X,-Y)=M_+.$$
%\end{cor}
%Here $M_+\in\Mat_2(\K [[\aaa,\bbb,\ccc-1]])$
%is the matrix 
%defined in the same way as %$M_{\varphi}$.
%%It is also same as 
%$M_\varphi$ in Notation \ref{notation for M}
%whose $\Gamma_\varphi(t)$ is replaced with
%$$ 
%\Gamma_+(t)%=\sqrt{\frac{t}{2}\mathrm{cosech}(\frac{t}{2})}
%=\sqrt{\Gamma(1+\frac{t}{2\pi \sqrt{-1}})\Gamma(1-\frac{t}{2\pi \sqrt{-1}})}
%=\sqrt{\frac{t}{e^\frac{t}{2}-e^{-\frac{t}{2}}}}
%=\exp\{-\sum_{n=1}^\infty \frac{B_{2n}}{2(2n)!}t^{2n}\}$$
%(N.B. the symbol $\Gamma$ means the classical gamma function (cf. Remark \ref{rem:Gamma}. (1))
%
%\begin{proof}

\bigskip
\noindent
{\bf Proof of Theorem \ref{thm:even unitary associator=M+0}.}
For any even unitary associator $\varphi$, %(cf. \cite{EF1} Lemma 9.5).
we have $\Gamma_{\varphi}(t)=\Gamma_+(t)$
by Remark \ref{rem:Gamma}.(2),
and thus,
\begin{equation}\label{eq: M+=M for even associator}
M_+=M_\varphi
\end{equation}
Therefore,  we see that our claim is a direct consequence of Theorem \ref{thm: varphi=M}.
\qed
%\end{proof}

\bigskip
We note that
\begin{equation}\label{eq: M+ in SL2}
M_+\in\SL_2(\Q[[\aaa,\bbb,\ccc-1]])
\end{equation}
by Proposition \ref{prop: M in SL2} and Theorem \ref{thm:even unitary associator=M+0}.

%\begin{defn}
%Let us take $\K=\Q_\ell$.
%In the same way as ${\Theta_0}$ in \eqref{eq:Theta0},
%we define 
We consider the group homomorphism
$$\Theta:\hat F_2^{(\ell)}\to\GL_2(\Q_\ell [[\aaa,\bbb,\ccc-1]])$$
defined by the evaluation
$x_0\mapsto e^X$ and $x_1\mapsto M_{+}^{-1}e^{-Y}M_{+}$. 
%\end{defn}
By the map $\Theta$,
our $\ell$-adic hypergeometric function is 
reformulated as follows: 
%by replacing roman letters with greek letters in the corresponding definitions.
%We put
\begin{prop}\label{prop: HG=Theta}
Definition \eqref{eq:HG and even associator}
is  free from any choice of  $\ell$-adic even unitary associator $\varphi$.
Furthermore the following equality holds:
\begin{equation}\label{eq:formal HG function}
\HG{\aaa,\bbb}{\ccc}{z}(\sigma)=[{\Theta}(f_\sigma^z)]_{11}
\in\Q_\ell[[\aaa,\bbb,\ccc-1]].
\end{equation}
\end{prop}

\begin{proof}
From Theorem \ref{thm:even unitary associator=M+0},
for even unitary associator $\varphi$,
we learn
\begin{equation}\label{eq:Theta=ev iota}
\Theta=\ev_{(X,-Y)}\circ \iota_{\varphi}
\end{equation}
where $\ev_{(X,-Y)}$ is the map defined in \eqref{eq: ev(X,-Y)}.
%with $\iota_{\varphi}=\iota_{(1,\varphi)}$
%from which we obtain the claim.
Whence it is evident that the definition is independent of $\varphi$.
\end{proof}

%Before we go to prove  Theorem \ref{thm: l-adic Gauss's hypergeometric theorem},
The following two propositions are required  to prove  the convergence: % of the matrix $M_{+,0}$.

\begin{prop}\label{convergence for M (1)}
When $a,b,c\in\Z_\ell$ satisfy \eqref{eq:ABC condition},
the evaluation $M_{+,0}$ % in \eqref{eq:M+0} is well-defined in
of the matrix $M_+$ at $(\aaa,\bbb,\ccc)=(a,b,c)$
makes sense in $\SL_2(\Q_\ell)$,
belongs to $\SL_2(\Z_\ell)$ and satisfies the congruence
$M_{+,0}\equiv I_2 \pmod \ell$.
%And there is 
%%the continuous representation $\Theta_0$ in \eqref{eq:Theta0}
%%is deduced from the evaluation
%%$x_0\mapsto \exp(X_0)$ and $x_1\mapsto M_{+,0}^{-1}\exp(-Y_0)M_{+,0}$.
%a continuous  homomorphism
%\begin{equation*}\label{eq:Theta0}
%\Theta_0:\hat F_2^{(\ell)}\to\GL_2(\Z_\ell)
%\end{equation*}
%%is deduced from the evaluation 
%sending
%$x_0\mapsto \exp(X_0)$ and $x_1\mapsto M_{+,0}^{-1}\exp(-Y_0)M_{+,0}$ 
%%(proven in Proposition \ref{convergence for M}).
%with 
%$$X_0=
%\begin{pmatrix}
%0 & b \\
%0 & p
%\end{pmatrix}, \
%Y_0=
%\begin{pmatrix}
%0 & 0 \\
%a & q
%\end{pmatrix}\in \Mat_2(\Z_\ell),
%$$
%with $p=1-c$ and $q=a+b+1-c=a+b+p$.
\end{prop}

\begin{proof}
We show the claim entry-wise.

{\it The $(1,1)$ entry:}
%We first prove that 
%$\Theta_0$ is a continuous map to $\GL_2(\Z_\ell)$.
%$\Gamma_+(t)$ converges when $|t|_\ell<1$.
By von Staudt-Clausen's theorem, $\ell B_{2n}\in\Z_\ell$ for all $n$.
Since we have $v_\ell(n!)<\frac{n}{\ell-1}$ for all $n$,
we have 
%$v_\ell(-\sum_{n=1}^\infty \frac{B_{2n}}{2(2n)!}t^{2n})\geqslant 1$.
\begin{equation}\label{eq:valuation}
v_\ell(\frac{B_{2n}}{2(2n)!}\ell^{2n})> 2n\frac{\ell -2}{\ell -1}-1>0
\end{equation}
where $v_\ell$ is the standard $\ell$-adic valuation.
Therefore, $\Gamma_+(\ell z)$ is in the Tate algebra $T_1$, that is,
it is a rigid-analytic function on 
$|z|_\ell\leqslant 1$ (cf. \cite{BGR}).
Therefore, $\GGamma{+}{-\ell\ppp,-\ell\qqq}{-\ell\ppp-\ell\aaa,-\ell\ppp-\ell\bbb}$
is in the Tate algebra $T_3$ with respect to three variables
$\aaa$, $\bbb$, $\ppp$.
%$\aaa^\dag=\frac{\aaa}{\ell}$,
%$\bbb^\dag=\frac{\bbb}{\ell}$,
%$\ppp^\dag=\frac{\ppp}{\ell}$.
Hence, %the $(1,1)$-component
$[M_{+,0}]_{11}:=\GGamma{+}{-p,-q}{-p-a,-p-b}$ makes sense in $\Q_\ell$ when \eqref{eq:ABC condition} holds.
Furthermore, by \eqref{eq:valuation},
we have $\log\Gamma_+(\ell z)\in \ell\Z_\ell[[z]]$,
and whence $\Gamma_+(\ell z)\in 1+\ell\Z_\ell[[z]]$.
Thus,
$\GGamma{+}{-\ell\ppp,-\ell\qqq}{-\ell\ppp-\ell\aaa,-\ell\ppp-\ell\bbb}\in
1+\ell\Z_\ell[[\aaa,\bbb,\ppp]]$.
Thus, $[M_{+,0}]_{11}$ belongs to $1+\ell \Z_\ell$.

{\it The $(1,2)$ entry:}
%By Theorem \ref{thm: varphi=M},
%Since we have $M_+\in\GL_2(\Q[[\aaa,\bbb,\ccc-1]])$,
The above arguments indicate that
$\GGamma{+}{\ell\ppp,-\ell\qqq}{-\ell\aaa,-\ell\bbb}$ and
$\GGamma{+}{-\ell\ppp,-\ell\qqq}{-\ell\ppp-\ell\aaa,-\ell\ppp-\ell\bbb}$
are in $T_3$.
Write
$m:=\frac{\bbb}{\ppp}\left\{\GGamma{+}{\ell\ppp,-\ell\qqq}{-\ell\aaa,-\ell\bbb}-\GGamma{+}{-\ell\ppp,-\ell\qqq}{-\ell\ppp-\ell\aaa,-\ell\ppp-\ell\bbb}
\right\}$.
By \eqref{eq: M+ in SL2},
$m$ is in $\Q_\ell[[\aaa,\bbb,\ccc-1]]$.
Hence, by the Weierstrass division theorem (cf. \cite{BGR}),
%$[M_{+}]_{(1,2)}$
we see that $m$ is in  $T_3$.
Thus, the entry $[M_{+,0}]_{12}$ makes sense in $\Q_\ell$.
Furthermore, by \eqref{eq:valuation},
$m$ is also in $\ell\Z_\ell[[\aaa,\bbb,\ppp]]$.
Whence we learn that  $[M_{+,0}]_{12}\in \ell \Z_\ell$.

{\it The $(2,1)$ entry:}
By similar arguments as the above, we have
$[M_{+,0}]_{(2,1)}\in \ell \Z_\ell$.
%in the same way to the above. 

{\it The $(2,2)$ entry:}
%Since any $\varphi$ is commutator group-like by definition, $M_\varphi$ is in 
%$\SL_2(\Q_\ell[[\aaa,\bbb,\ccc-1]])$.
%Particularly 
By \eqref{eq: M+ in SL2}, $\det(M_+)=1$.
%$M_+$ is in $\SL_2(\Q_\ell[[\aaa,\bbb,\ccc-1]])$.
% by Corollary \ref{cor: varphi=M+}.
Thus, we have $M_{+,0}\in \SL_2(\Q_\ell)$ with $[M_{+,0}]_{22}\in 1+\ell \Z_\ell$
because 
 we have shown that $[M_{+,0}]_{11}\in 1+\ell \Z_\ell$
and $[M_{+,0}]_{12}, [M_{+,0}]_{21}\in \ell \Z_\ell$.
%we get $[M_{+,0}]_{22}\in 1+\ell \Z_\ell$.
%%By \eqref{eq:valuation},  $\Gamma_+(\ell z)$ is invertible in $T_1$ and so for
%$\GGamma{+}{-\ell\ppp,-\ell\qqq}{-\ell\ppp-\ell\aaa,-\ell\ppp-\ell\bbb}$
%in $T_3$.
%It follows that the  $(1,1)$-entry  $[M_{+}]_{11}$ is invertible.
%%Therefore we have $[M_{+}]_{22}$ in $T_3$ because
%%$[M_{+}]_{12}$ and  $[M_{+}]_{21}$ are in $T_3$ and 
%By $\det M_+=1$,
%we see that  the entry $[M_{+,0}]_{22}$ makes sense.

%The second claim is shown as follows:
%There is a continuous inclusion, called the Magnus embedding,
%$\hat F_2^{(\ell)}\hookrightarrow \Z_\ell\langle\langle u,v\rangle\rangle$
%sending $x_0$ and $x_1$ to $1+u$ and $1+v$ respectively (\cite{Serre} Ch I.\S 1).
%Our map $\Theta_0$ is obtained by composing it with 
%the map $ \Z_\ell\langle\langle u,v\rangle\rangle\to \Mat_2(\Z_\ell)$
%sending $u$ and $v$ to $\exp(X_0)-I_2$ and $M_{+,0}^{-1}\exp(-Y_0)M_{+,0}-I_2$.
%The claim follows because it is continuous by 
%\begin{equation}\label{eq:cont bmod ell}
%\exp(X_0)\equiv \exp(-Y_0)\equiv M_{+,0}\equiv I_2
%\bmod \ell.
%\end{equation}
\end{proof}

\begin{prop}\label{convergence for M (2)}
When $a,b,c\in\Z_\ell$ satisfy \eqref{eq:ABC condition},
%the evaluation $M_{+,0}$ % in \eqref{eq:M+0} is well-defined in
%of the matrix $M_+$ at $(\aaa,\bbb,\ccc)=(a,b,c)$
%makes sense in 
%$\SL_2(\Z_\ell)$ and it satisfies the congruence
%$M_{+,0}\equiv I_2 \pmod \ell$.
%And 
there is 
%the continuous representation $\Theta_0$ in \eqref{eq:Theta0}
%is deduced from the evaluation
%$x_0\mapsto \exp(X_0)$ and $x_1\mapsto M_{+,0}^{-1}\exp(-Y_0)M_{+,0}$.
a continuous group homomorphism
\begin{equation*}\label{eq:Theta0}
\Theta_0:\hat F_2^{(\ell)}\to\GL_2(\Z_\ell)
\end{equation*}
%is deduced from the evaluation 
sending
$x_0\mapsto \exp(X_0)$ and $x_1\mapsto M_{+,0}^{-1}\exp(-Y_0)M_{+,0}$ 
%(proven in Proposition \ref{convergence for M}).
with 
$$X_0=
\begin{pmatrix}
0 & b \\
0 & p
\end{pmatrix}, \
Y_0=
\begin{pmatrix}
0 & 0 \\
a & q
\end{pmatrix}\in \Mat_2(\Z_\ell),
$$
with $p=1-c$ and $q=a+b+1-c=a+b+p$.
\end{prop}

\begin{proof}
%The second claim is shown as follows:
There is a continuous inclusion, called the Magnus embedding,
$\hat F_2^{(\ell)}\hookrightarrow \Z_\ell\langle\langle u,v\rangle\rangle$
sending $x_0$ and $x_1$ to $1+u$ and $1+v$, respectively (\cite{Serre} Ch I.\S 1).
Our map $\Theta_0$ is obtained by composing it with 
the map $ \Z_\ell\langle\langle u,v\rangle\rangle\to \Mat_2(\Z_\ell)$
sending $u$ and $v$ to $\exp(X_0)-I_2$ and $M_{+,0}^{-1}\exp(-Y_0)M_{+,0}-I_2$.
By 
\begin{equation}\label{eq:cont bmod ell}
\exp(X_0)\equiv \exp(-Y_0)\equiv M_{+,0}\equiv I_2
\bmod \ell,
\end{equation}
$\Theta_0$ is well-defined and continuous.
\end{proof}

%\bigskip
%\noindent
%{\bf Proof of Theorem \ref{thm:even unitary associator=M+0}.}
%By the proof of the first claim of Proposition \ref{convergence for M},
%we see that  the matrix $M_{+,0}$ converges.
%Whence by Corollary \ref{cor: varphi=M+},
%we get the proof. 
%\qed

\bigskip
\noindent
{\bf Proof of Theorem \ref{fundamental theorem of l-adic HG}.}
Claim (i) follows from Proposition \ref{prop: HG=Theta}.
Claim (ii) is a consequence of  Proposition \ref{convergence for M (2)}
because we have
$$
\HG{a,b}{c}{z}(\sigma)=[{\Theta_0}(f_\sigma^z)]_{11}.
$$
\qed
\bigskip

Let $\sigma\in G_\Q$, which corresponds to 
$(\lambda_\sigma, f_\sigma)\in \GT(\Q_\ell)$ under \eqref{eq: GQ to GT}.
%Similarly to  $\Gamma_\varphi$, 
We define 
%\begin{equation}\label{eq:Gamma sigma}
%\Gamma_{\sigma}(t):=\exp\left\{\sum_{n=1}^\infty 
%\frac{(-1)^{n+1}}{n}
%\left(f_\sigma(e^{e_0},e^{e_1})|e_0^{n-1}e_1\right)t^n\right\}
%\in\Q_l[[t]]
%\end{equation}
%and
$\GGamma{\sigma}{s,t}{u,v}:=
\frac{\Gamma_{\sigma}(s)\Gamma_{\sigma}(t)}{\Gamma_{\sigma}(u)\Gamma_{\sigma}(v)}$
with $\Gamma_{\sigma}(t)=\Gamma_{f_\sigma}(t)$.
Then, a formal version of our $\ell$-adic Gauss hypergeometric theorem
is given as follows:

\begin{thm}%[$\ell$-adic hypergeometric theorem]
\label{thm: formal l-adic Gauss's hypergeometric theorem}
%Let $a,b,c\in\Z_\ell$ satisfying \eqref{eq:ABC condition}.
For $\sigma\in G_{\Q}$, the following equality holds in
$\Q_\ell[[\aaa,\bbb,\ccc-1]]$:
{\small
\begin{align}
&\HG{\aaa,\bbb}{\ccc}{\vec{10}}(\sigma)=
\frac{\ppp\qqq}{\aaa\bbb}\left\{\GGamma{+}{-\ppp,-\qqq}{-\ppp-\aaa,-\ppp-\bbb}+\GGamma{+}{\ppp,-\qqq}{-\aaa,-\bbb}\right\}\GGamma{+}{-\ppp,\qqq}{\aaa,\bbb}\cdot\GGamma{\sigma}{-\ppp,\qqq}{\aaa,\bbb}+ \\ \notag
&\left\{\frac{\aaa\bbb+\ppp\qqq}{\ppp\qqq}\GGamma{+}{\ppp,\qqq}{\ppp+\aaa,\ppp+\bbb}
-\frac{\aaa\bbb}{\ppp\qqq}\GGamma{+}{-\ppp,\qqq}{\aaa,\bbb}\right\}
\GGamma{+}{-\ppp,-\qqq}{-\ppp-\aaa,-\ppp-\bbb}\cdot
\GGamma{\sigma}{-\ppp,-\qqq}{-\ppp-\aaa,-\ppp-\bbb}.
\end{align}
}
\end{thm}

\begin{proof}
Take any even unitary associator $\varphi$. 
Write
$(\mu',\varphi')=(\lambda_\sigma,f_\sigma)\circledast (1,\varphi)$, that is,
%By definition, we have
$$
(\mu',\varphi')
=\left(\lambda_\sigma,\ \varphi\cdot f_\sigma(e^{e_0},\varphi^{-1}e^{e_1}\varphi)\right).
$$
Then the  pair forms an associator
because $M(\Q_\ell)$ is a $\GT(\Q_\ell)$-torsor
(cf. Definition \ref{defn:associator}. (3)).
Thus, we have
\begin{align*}
\HG{{\aaa},{\bbb}}{{\ccc}}{\vec{10}}(\sigma)
&=\left[\Theta(f_\sigma)\right]_{11}
=\left[\ev_{(X,-Y)}(f_\sigma(e^{e_0},\varphi^{-1}e^{e_1}\varphi))\right]_{11} \\
&=\left[\ev_{(X,-Y)}(\varphi^{-1}\cdot\varphi')\right]_{11} 
=\left[\varphi(X,-Y)^{-1}\cdot\varphi'(X,-Y)\right]_{11} \\
&=\left[M_+^{-1}\cdot M_{\varphi'}\right]_{11}
\end{align*}
by Theorem \ref{thm: varphi=M} and 
\eqref{eq: M+=M for even associator}.
%Corollary \ref{cor: varphi=M+}.
%By \cite{EF3} Remark 4.6,
%we  have
%$\Gamma_{\varphi'}(t)=\Gamma_\varphi(t)\Gamma_\sigma(t)
%=\Gamma_+(t)\Gamma_\sigma(t)$.
By calculating the $(1,1)$ entry of the  above matrix
with \eqref{eq: Gamma composition},
we obtain the claim.
\end{proof}

\smallskip
The $\ell$-adic Gauss hypergeometric theorem can be derived as follows:
%\smallskip

\noindent
{\bf Proof of
Theorem \ref{thm: l-adic Gauss's hypergeometric theorem}.}
It is a consequence of
Proposition \ref{convergence for M (2)} and
Theorem \ref{thm: formal l-adic Gauss's hypergeometric theorem}.
%By Proposition \ref{convergence for M},
%the map $\Theta_0$ is realized as the evaluation of $\Theta$ at
%$\aaa=a$, $\bbb=b$, $\ccc=c$. 
%%By  $|a|_\ell$, $|b|_\ell$, $|p|_\ell$, $|q|_\ell<1$,
%%%we see that 
%%both
%%$\exp(X_0)$ and $\exp(-Y_0)$ converge %in $I_2+\Mat_2(\ell\Z_\ell)$.
%%%Whence  we have $M_{+,0}$ 
%%in $\Mat_2(\Q_\ell)$.
%%By Proposition \ref{convergence for M},
%%we have $M_{+,0}\in\GL_2(\Q_\ell)$.
%%So $\Theta_0$ is realized as the evaluation of $\Theta$ at
%%$\aaa=a$, $\bbb=b$, $\ccc=c$.
%Whence the result follows from Theorem \ref{thm: formal l-adic Gauss's hypergeometric theorem}.
\qed

%%%%%%%%%%%%%%%%%%%%%%%%%%%%%%%%%%%%%%%%%%%%%%%%%%%%%%%%%%%%%%%%%%%%%%%
\section{$\ell$-adic Euler transformation formula}
\label{sec:Euler's transformation formula}
In this section, 
we introduce two more series
in Definition \ref{def:two 2F1}
to relax %the domain of $\HG{a,b}{c}{z}(\sigma)$
our assumption
\eqref{eq:ABC condition} 
%to  $(a,b,c)\in\mathfrak D$
in Proposition \ref{prop:enlarge}
%in the definition of the $\ell$-adic hypergeometric function  
%$\HG{a,b}{c}{z}(\sigma)$
%by introducing $\ell$-adic analogues of six solutions (\S \ref{sec:complex case})
%of the hypergeometric equation
and to provide  a proof of the $\ell$-adic Euler transformation formula
(Theorem \ref{thm: l-adic Euler's transformation formula}).
%and of Pfaff's one 
%(Theorem \ref{thm: l-adic Pfaff's transformation formula}).

In this section, we fix an even unitary associator $\varphi$.

\begin{prop}\label{prop:even associator free}
%If $(\mu,\varphi)$ is an even associator,
All six matrices $\V_{\ast}^{\varphi}(\sigma)(z)$
with $\ast=\vec{01}$, $\vec{10}$, $\vec{1\infty}$, $\vec{\infty 1}$, $\vec{\infty 0}$
and $\vec{0\infty}$
%The elements
%$\V_{\vec{01}}^{\varphi}(\sigma)(z)$,
%$\V_{\vec{10}}^{\varphi}(\sigma)(z)$,
%$\V_{\vec{0\infty}}^{\varphi}(\sigma)(z)$
%the images of the above $5$ elements under $\ev_{(X,-Y)}$
are independent of any choice of even unitary associators
$\varphi$.
\end{prop}

\begin{proof}
\begin{enumerate}
\renewcommand{\labelenumi}{(\roman{enumi}).}
\item
The claim for $\V_{\vec{01}}(\sigma)(z)$
is a consequence of 
\begin{equation}\label{eq:G01=Theta}
G_{\vec{01}}^{\varphi}(X,-Y)(\sigma)(z)=
\Theta(f_\sigma^z) %G_{\vec{01}}^{\varphi}(e_0,e_1)(\sigma)(z)).
\end{equation}
which follows from \eqref{eq:Theta=ev iota}.
%The claim for $\V_{\vec{0\infty}}(\sigma)(z)$
%follows from \eqref{eq:G01=Theta} and
%$$
%G_{\vec{0\infty}}^{\varphi}(X,-Y)(\sigma)(z)=
%G_{\vec{01}}^{\varphi}(X,-Y)(\sigma)(z)\cdot
%\Theta(x_0^{m_\sigma})
%$$
%which follows from \eqref{eq:G=G last}.
%We can show the claim for the other four cases in a same way
%by using \eqref{eq:G=G first}-\eqref{eq:G=G last}.
%, \eqref{eq:Theta=ev iota} and
%\eqref{eq:G01=Theta},
%we see that the other five $G^{\varphi}(X,-Y)(\sigma)(z)$'s 
%are independent of any choice of even unitary associators.
%Hence we get the claim.
%
%\bigskip
\item
The claim for $\V_{\vec{10}}^{\varphi}(\sigma)(z)$
follows from
\begin{align}\notag
G_{\vec{10}}^{\varphi}(X,-Y)(\sigma)(z)
&=\ev_{(X,-Y)}\left(G_{\vec{10}}^{\varphi}(e_0,e_1)(\sigma)({z})\right) %\\ \notag
=\ev_{(X,-Y)}\left(G_{\vec{01}}^{\varphi}(e_1,e_0)(\sigma)(1-{z}) \right)\\ \notag
&=\ev_{(X,-Y)}(f_{\sigma}^{1-z}\left(\exp(e_1),\varphi(e_1,e_0)^{-1}\exp(e_0)\varphi(e_1,e_0)\right).  \notag
\intertext{By 2-cycle relation for $\varphi$,} \notag
%&=G_{\vec{01}}^{\varphi}(-Y,X)(\sigma)(1-z) \\ \notag
%&=\ev_{(-Y,X)}\left(f_{\sigma}^{1-z}(\exp(e_0),\varphi^{-1}\exp(e_1)\varphi)\right) \\ 
&=\ev_{(X,-Y)}(f_{\sigma}^{1-z}\left(\exp(e_1),\varphi(e_0,e_1)\exp(e_0)\varphi(e_0,e_1)^{-1}\right) \\
\label{eq:G10=f}
&=f_{\sigma}^{1-z}\left(\exp(-Y),M_+\exp(X)M_+^{-1}\right).
\end{align}
where in the last equality we use
%the 2-cycle relation of $\varphi$
%$\varphi(e_0,e_1)\varphi(e_1,e_0)=1$ 
%and 
%Corollary \ref{cor: varphi=M+}.
Theorem \ref{thm:even unitary associator=M+0}.
%
%\bigskip
\item
The claim for $\V_{\vec{1\infty}}^{\varphi}(\sigma)(z)$ follows from
the following equality:
\begin{align*} \notag
G_{\vec{1\infty}}^{\varphi}&(X,-Y)(\sigma)(z)
=\ev_{(X,-Y)}\left(G_{\vec{1\infty}}^{\varphi}(e_0,e_1)(\sigma)({z})\right) %\\ \notag
=\ev_{(X,-Y)}\left(G_{\vec{01}}^{\varphi}(e_1,e_\infty)(\sigma)(1-\frac{1}{z}) \right)\\ \notag
&=\ev_{(X,-Y)}\left(f_{\sigma}^{1-\frac{1}{z}}\left(\exp(e_1),\varphi(e_1,e_\infty)^{-1}\exp(e_\infty)\varphi(e_1,e_\infty)\right)\right).  \notag
\intertext{By 2- and 3-cycle relations for $\varphi$,} 
&=\ev_{(X,-Y)}\left(f_{\sigma}^{1-\frac{1}{z}}\left(\exp(e_1),\exp(-\frac{e_1}{2})\varphi(e_0,e_1)\exp(-e_0)\varphi(e_0,e_1)^{-1}\exp(-\frac{e_1}{2})
\right)\right) \\ \notag
&=f_{\sigma}^{1-\frac{1}{z}}\left(\exp(-Y),
\exp(Y/2)M_+\exp(-X)M_+^{-1}\exp(Y/2)\right).
\end{align*}
%
%\bigskip
\item
The claim for $\V_{\vec{0\infty}}^{\varphi}(\sigma)(z)$ follows from
the following equality:
\begin{align} \notag
G_{\vec{0\infty}}^{\varphi}&(X,-Y)(\sigma)(z)
=\ev_{(X,-Y)}\left(G_{\vec{0\infty}}^{\varphi}(e_0,e_1)(\sigma)({z})\right) 
%\\ \notag
%&
=\ev_{(X,-Y)}\left(G_{\vec{01}}^{\varphi}(e_0,e_\infty)(\sigma)({z}/(z-1))\right)
 \\ \notag
%&=G_{\vec{01}}^{\varphi}(X,Y-X)(\sigma)({z}/(z-1)) \\ \notag
%&=\ev_{(X,Y-X)}\left(f_{\sigma}^{z/(z-1)}(\exp(e_0),\varphi^{-1}\exp(e_1)\varphi)\right)  \\ \notag
&=\ev_{(X,-Y)}\left(f_{\sigma}^{z/(z-1)}(\exp(e_0),\varphi(e_0,e_\infty)^{-1}\exp(e_\infty)\varphi(e_0,e_\infty))\right). \notag
\intertext{By 2- and 3-cycle relations for $\varphi$,} \notag
&=\ev_{(X,-Y)}\left(f_{\sigma}^{z/(z-1)}(\exp(e_0),
\exp(-e_0/2)\varphi(e_0,e_1)^{-1}\exp(-e_1)\varphi(e_0,e_1)\exp(-e_0/2)
)\right)  \\  \label{eq:G0infty=f}
&=f_{\sigma}^{z/(z-1)}\left(\exp(X),\exp(-X/2)M_+^{-1}\exp(Y)M_+\exp(-X/2)\right).
\end{align}
%where in the fourth equality we use the 2- and 3-cycle relation for $\varphi$.
%\Add{Give a proof}
%\Erase{by introducing a new matrix $N$ by using \cite{O}(39).}
%\Erase{Need calculate $\ev_{(X,-Y)}(\varphi(e_0,e_\infty))$.}
%
%
%\bigskip
\item
The claim for $\V_{\vec{\infty 1}}^{\varphi}(\sigma)(z)$ follows from
the following equality:
\begin{align} \notag
G_{\vec{\infty 1}}^{\varphi}&(X,-Y)(\sigma)(z)
=\ev_{(X,-Y)}\left(G_{\vec{\infty 1}}^{\varphi}(e_0,e_1)(\sigma)({z})\right) 
%\\ %\notag
%&
=\ev_{(X,-Y)}\left(G_{\vec{01}}^{\varphi}(e_\infty,e_1)(\sigma)(\frac{1}{z})\right)
 \\ \notag
&=\ev_{(X,-Y)}\left(f_{\sigma}^{1/z}(\exp(e_\infty), \varphi(e_\infty,e_1)^{-1}\exp(e_1)\varphi(e_\infty,e_1))\right) \\ \notag
&=f_{\sigma}^{1/z}\left(\exp(Y-X),N_+^{-1}\exp(-X)N_+\right) \notag
\end{align}
with $N_+:=\ev_{(X,-Y)}(\varphi(e_\infty,e_1))$,
which is shown to be free from any choice of even unitary associator $\varphi$
in Lemma \ref{lemma in appendix}.
%
%\bigskip
\item
The claim for $\V_{\vec{\infty 0}}^{\varphi}(\sigma)(z)$  follows from
the following equality:
\begin{align} \notag
G_{\vec{\infty 0}}^{\varphi}&(X,-Y)(\sigma)(z)
=\ev_{(X,-Y)}\left(G_{\vec{\infty 0}}^{\varphi}(e_0,e_1)(\sigma)({z})\right) 
%\\ %\notag
%&
=\ev_{(X,-Y)}\left(G_{\vec{01}}^{\varphi}(e_\infty,e_0)(\sigma)(\frac{1}{1-z})\right)
 \\ \notag
&=\ev_{(X,-Y)}\left(f_{\sigma}^{1/(1-z)}(\exp(e_\infty), \varphi(e_\infty,e_0)^{-1}\exp(e_0)\varphi(e_\infty,e_0))\right).  \notag
\intertext{By 2- and 3-cycle relations for $\varphi$,} \notag
&=\ev_{(X,-Y)}\left(f_{\sigma}^{1/(1-z)}(\exp(e_\infty),
\exp(-\frac{e_\infty}{2}) \varphi(e_\infty,e_1)^{-1}\exp(-e_1)\varphi(e_\infty,e_1)
\exp(-\frac{e_\infty}{2})
)\right) \\ \notag
&=f_{\sigma}^{1/(1-z)}\left(\exp(Y-X),\exp(-\frac{Y-X}{2})N_+^{-1}\exp(Y)N_+\exp(-\frac{Y-X}{2})\right). \notag
\end{align}
\end{enumerate}
\end{proof}

%By Proposition \ref{prop:even associator free}, the above $6$ matrices are independent of any choice of even associators $(\mu,\varphi)$.
Henceforth, we drop the upper suffix  $\varphi$
in $\V_{\ast}^{\varphi}(\sigma)(z)$.
%and simply denote by $\V_{\ast}(\sigma)(z)$.
%denote the elements simply by
%$\V_{\vec{01}}(\sigma)(z)$,
%$\V_{\vec{10}}(\sigma)(z)$,
%$\V_{\vec{0\infty}}(\sigma)(z)$.
%\in \GL_2(\PP)$.
%when we take  an even unitary associator $\varphi$.
%We learn from the above equality that 
%the left hand side is independent of any choice of even unitary associators.
By \eqref{eq:PhiG01}, \eqref{eq:formal HG function},
\eqref{eq:G01=Theta} and
$G^\varphi_{\vec{01}}(X,-Y)(\sigma)(z)\in \GL_2(\PP)$,
we see that the $(1,1)$ entry
$[\V_{\vec{01}}(\sigma)(z)]_{11}$ is given by
\begin{equation}\label{eq:Phi01}
[\V_{\vec{01}}(\sigma)(z)]_{11}=\HG{\aaa,\bbb}{\ccc}{z}(\sigma)
\in \PP=\Q_\ell[[\aaa,\bbb,\ccc-1]].
\end{equation}
%Whence we have
%\begin{equation*}%\label{eq:HG and even associator}
%\HG{\aaa,\bbb}{\ccc}{z}(\sigma)=[\iota_\varphi(f_\sigma^z)(X,-Y)]_{(1,1)}
%\end{equation*}
%for any even unitary associator $\varphi$.

Recall that in the complex case (\S \ref{sec:complex case}), we
have the following relationship between the $(1,1)$ entry of the matrix constructed from the fundamental solution of the KZ-equation and the hypergeometric function
\begin{align*}
[\V_{\vec{10}}(z)]_{11}
&=\HG{a,b}{a+b+1-c}{1-z}, \\
[\V_{\vec{0\infty}}(z)]_{11}
&=(1-z)^{-a}\cdot
\HG{a,c-b}{c}{\frac{z}{z-1}},
\end{align*}
which indicates that we should introduce the following series in our $\ell$-adic setting:

\begin{defn}\label{def:two 2F1}
We consider two series
\begin{align}\label{eq:2F1prime}
&{\HGdag{\aaa',\bbb'}{\ccc'}{z}}(\sigma)\in {\ccc'}^{-1}{\Q}_\ell[[{\aaa'},{\bbb'},{\ccc'}]], %\quad
\\ \label{eq:2F1primeprime}
&{\HGddag{\aaa'',\bbb''}{\ccc'' }{z}}(\sigma)\in \Q_\ell[[\aaa''-1,\bbb'',\ccc''-1]]
\end{align}
which are determined by $\ell$-adic analogues of
the above two equalities:
%by using the
%associated with the $(1,1)$ entry of the  matrices $\V_{\vec{10}}(\sigma)(z)$ and
%$\V_{\vec{0\infty}}(\sigma)(z)$ as follows:
%{\small
\begin{align}
\label{eq:Phi10}
[\V_{\vec{10}}(\sigma)(z)]_{11}
&=:\HGdag{\aaa,\bbb}{\aaa+\bbb+1-\ccc}{1-z}(\sigma)
\in \qqq^{-1}\PP%\Q_\ell[[\aaa,\bbb,\ccc-1]]
,\\
%\label{eq:Phi1infty}
%[\V_{\vec{1\infty}}(\sigma)(z)]_{11}
%&=:\exp\left\{-\rho_{z}(\sigma)\aaa\right\}\cdot
%\HG{\aaa,\aaa+1-\ccc}{\aaa+\bbb+1-\ccc}{1-\frac{1}{z}}(\sigma)
%%\in \Q_\ell[[\aaa,\bbb,\ccc-1]]
%,\\
%\label{eq:Phiinfty1}
%[\V_{\vec{\infty 1}}(\sigma)(z)]_{11}
%&=:\exp\left\{-\rho_{z}(\sigma)\aaa\right\}\cdot
%\HG{\aaa+1-\ccc,\aaa}{\aaa-\bbb+1}{\frac{1}{z}}(\sigma)
%%\in \Q_\ell[[\aaa,\bbb,\ccc-1]]
%,\\
\label{eq:Phi0infty}
[\V_{\vec{0\infty}}(\sigma)(z)]_{11}
&=:
\exp\left\{-\rho_{1-z}(\sigma)\aaa\right\}\cdot
\HGddag{\aaa,\ccc-\bbb}{\ccc}{\frac{z}{z-1}}(\sigma)
\in \PP.
%\Q_\ell[[\aaa,\bbb,\ccc-1]].
%,\\
%\label{eq:Phiinfty0}
%[\V_{\vec{\infty 0}}(\sigma)(z)]_{11}
%&=:\exp\left\{-\rho_{1-z}(\sigma)\aaa\right\}\cdot
%\HG{\ccc-\bbb,\aaa}{\aaa-\bbb+1}{\frac{1}{1-z}}(\sigma)
%\in \Q_\ell[[\aaa,\bbb,\ccc-1]].
\end{align}
%}%They are all elements in $\Q_\ell[[\aaa,\bbb,\ccc-1]]$.
\end{defn}
We stress that 
by Proposition \ref{prop:even associator free},
the two series are independent of any choice of even unitary associators $\varphi$.
%The following says that the above definition is well-defined under the change of variables:

We note that the relationship of
the  three series
$\HG{\aaa,\bbb}{\ccc}{z}(\sigma)\in \Q_\ell[[\aaa,\bbb,\ccc-1]]$,
$\HGdag{\aaa',\bbb'}{\ccc'}{z}(\sigma)\in {\ccc'}^{-1}{\Q}_\ell[[{\aaa'},{\bbb'},{\ccc'}]]$, 
$\HGddag{\aaa'',\bbb''}{\ccc'' }{z}(\sigma)\in
\Q_\ell[[\aaa''-1,\bbb'',\ccc''-1]]$
%are also related to the other three matrixes
with the other three solutions
$\V_{\vec{\infty 1}}(\sigma)(z)$,
$\V_{\vec{1\infty}}(\sigma)(z)$,
$\V_{\vec{\infty 0}}(\sigma)(z)$
is given as follows:

\begin{prop}\label{three compatibilities}
%(i). 
%The equations  \eqref{eq:Phi01} and  \eqref{eq:Phiinfty1} determine a well-defined element  
%$\HG{\aaa,\bbb}{\ccc}{z}(\sigma)\in \Q_\ell[[\aaa,\bbb,\ccc-1]]$.
%%\begin{equation}
%%\label{eq:Phiinfty1}
%%[\V_{\vec{\infty 1}}(\sigma)(z)]_{11}
%%=\exp\left\{-\rho_{z}(\sigma)\aaa\right\}\cdot
%%\HG{\aaa+1-\ccc,\aaa}{\aaa-\bbb+1}{\frac{1}{z}}(\sigma)
%%\in \Q_\ell[[\aaa,\bbb,\ccc-1]]
%%\end{equation}
%
%(ii). The equations  \eqref{eq:Phi10} and  \eqref{eq:Phi1infty} determine 
%a well-defined element
%$\HG{\aaa',\bbb'}{\ccc'}{z}(\sigma)\in \Q_\ell[[\aaa',\bbb',\ccc']]$.
%
%(iii). The equations  \eqref{eq:Phi0infty} and  \eqref{eq:Phiinfty0} determine
%a well-defined element  
%$\HG{\aaa'',\bbb''}{\ccc''}{z}(\sigma)\in \Q_\ell[[\aaa''-1,\bbb'',\ccc''-1]]$.
For an even unitary associator $\varphi$,
we have the following equalities: % hold in $\Q_\ell[[\aaa,\bbb,\ccc-1]]$:
\begin{align}
%\label{eq:Phi10}
%[\V_{\vec{10}}^{(\mu,\varphi)}(\sigma)(z)]_{11}
%&=:\HG{\bbb,\aaa}{\aaa+\bbb+1-\ccc}{1-z}(\sigma)
%\in \Q_\ell[[\aaa,\bbb,\ccc-1]]
%,\\
\label{eq:Phiinfty1}
[\V_{\vec{\infty 1}}(\sigma)(z)]_{11}
&=\exp\left\{-\rho_{z}(\sigma)\aaa\right\}\cdot
\HG{\aaa+1-\ccc,\aaa}{\aaa-\bbb+1}{\frac{1}{z}}(\sigma)
\text{ in } \PP
%\in \Q_\ell[[\aaa,\bbb,\ccc-1]]
,\\
\label{eq:Phi1infty}
[\V_{\vec{1\infty}}(\sigma)(z)]_{11}
&=\exp\left\{-\rho_{z}(\sigma)\aaa\right\}\cdot
\HGdag{\aaa,\aaa+1-\ccc}{\aaa+\bbb+1-\ccc}{1-\frac{1}{z}}(\sigma)
\text{ in } \qqq^{-1}\PP,  \\
%&=\Add{\HG{\bbb,\aaa}{\aaa+\bbb+1-\ccc}{1-\frac{1}{z}}(\sigma)}
%\in \Q_\ell[[\aaa,\bbb,\ccc-1]]
%,\\
%\label{eq:Phi0infty}
%[\V_{\vec{0\infty}}^{(\mu,\varphi)}(\sigma)(z)]_{11}
%&=:
%\exp\left\{-\rho_{1-z}^{(\mu,\varphi)}(\sigma)\aaa\right\}\cdot
%\HG{\ccc-\bbb,\aaa}{\ccc}{\frac{z}{z-1}}(\sigma)
%\in \Q_\ell[[\aaa,\bbb,\ccc-1]]
%,\\
\label{eq:Phiinfty0}
[\V_{\vec{\infty 0}}(\sigma)(z)]_{11}
&=\exp\left\{-\rho_{1-z}(\sigma)\aaa\right\}\cdot
\HGddag{\aaa,\ccc-\bbb}{\aaa-\bbb+1}{\frac{1}{1-z}}(\sigma)
\text{ in } \PP.
%\in \Q_\ell[[\aaa,\bbb,\ccc-1]].
\end{align}
\end{prop}

\begin{proof}
%Basically our methods of the proof are to deduce the validity from the ones in 
%the complex case via Chen's iterated integration theory. 

{\it The first equality:}
%Our claim says if we define
%$$
%\left[G_{\vec{01}}^{(\mu,\varphi)}(X,-Y)(\sigma)(z)\cdot
%\begin{pmatrix}
%1 & 1 \\
%0 & \frac{\ppp}{\bbb}
%\end{pmatrix}\right]_{11}
%=:\HG{\aaa,\bbb}{\ccc}{z}(\sigma),
%$$
%the following holds in $\Q_\ell[[\aaa,\bbb,\ccc-1]]$:
%$$
%\left[G_{\vec{0 1}}^{(\mu,\varphi)}(Y-X,-Y)(\sigma)(z)\cdot
%\begin{pmatrix}
%1 & 0 \\
%\frac{-\aaa}{\bbb} & -1
%\end{pmatrix}\right]_{11}
%=\exp\left\{\rho_{z}(\sigma)\aaa\right\}\cdot
%\HG{\aaa+1-\ccc,\aaa}{\aaa-\bbb+1}{z}(\sigma).
%$$ 
Let $\iota$ be the algebra automorphism of $\PP$
such that 
\begin{equation}\label{eq:iota-map}
\iota(\aaa)=\aaa,    \qquad
\iota(\bbb)=\aaa+1-\ccc, \qquad
\iota(\ccc-1)=\aaa-\bbb
\end{equation}
%sending
%$\aaa$, $\bbb$, $\ccc$ to
%$\aaa$, $\aaa+1-\ccc$, $\aaa-\bbb+1$ respectively.
By \eqref{eq:PhiGinfty1}, \eqref{eq:Phi01} and 
the replacement of $z$ with $z^{-1}$, we see that
it is sufficient to show that the equation
\begin{align*}
&\left[G^\varphi_{\vec{0 1}}(Y-X,-Y)(\sigma)(z)\cdot
\begin{pmatrix}
1 & 1 \\
\frac{-\aaa}{\bbb} & -1
\end{pmatrix}\right]_{11} \\
&\qquad \qquad =
\exp\left\{\rho_{z}(\sigma)\aaa\right\}\cdot
\iota\left(\left[G^\varphi_{\vec{01}}(X,-Y)(\sigma)(z)\cdot
\begin{pmatrix}
1 & 1 \\
0 & \frac{\ppp}{\bbb}
\end{pmatrix}\right]_{11}\right)
\end{align*}
holds in $\PP$.

In the complex case, we have
$
G_{\vec{0 1}}(Y-X,-Y)(z)=G_{\vec{\infty 1}}(X,-Y)(\frac{1}{z}).
$
By \cite{O} (36) and (97), we obtain
$$
\left[G_{\vec{0 1}}(Y-X,-Y)(z)\cdot
\begin{pmatrix}
1 & 1 \\
\frac{-\aaa}{\bbb} & -1
\end{pmatrix}\right]_{11}
=\exp\{\log(z)\aaa\}\cdot
\HG{\aaa,\aaa+1-\ccc}{\aaa-\bbb+1}{z},
$$
whence we obtain 
\begin{align}\label{eq:20210606}
&\left[G_{\vec{0 1}}(Y-X,-Y)(z)\cdot
\begin{pmatrix}
1 & 1 \\
\frac{-\aaa}{\bbb} & -1
\end{pmatrix}\right]_{11} \\ \notag
&\qquad \qquad =
\exp\left\{\log({z})\aaa\right\}\cdot
\iota\left(\left[G_{\vec{01}}(X,-Y)(z)\cdot
\begin{pmatrix}
1 & 1 \\
0 & \frac{\ppp}{\bbb}
\end{pmatrix}\right]_{11}\right)
\end{align}
in $\C[[\aaa,\bbb,\ccc-1]]$.
%Here $G_{\vec{01}}(e_0,e_1)(z)\in\C\langle\langle e_0,e_1\rangle\rangle$
%is the fundamental solution of the KZ-equation.
%It can be regarded as an element $G_{\vec{01}}(e_0,e_1)$ in
%$
%\mathrm{Map}\left(\mathcal P_{0}^{z}(\X(\C)), \C\right)\hat\otimes
%\C\langle\langle e_0,e_1\rangle\rangle.
%$

Let $TV=\oplus_{n=0}^\infty V^{\otimes n}$ with  $V:=H^1_{\mathrm{DR}}(\X,\Q)$
and $V^{\otimes 0}=\Q$,
where we encode a structure of Hopf algebra with the shuffle product and the deconcatenation coproduct.
We consider the $\Q$-linear map 
associated with iterated integrals
$$
\rho:TV
%\hookrightarrow
\to
\mathrm{Map}\left(\pi_1(\X(\C);0,z),\ \C
\right)
$$
which sends each $\omega_{i_m}\otimes\cdots\otimes\omega_{i_1}\in V^{\otimes m}$ to
\begin{equation*}%\label{iterated integral map}
\rho(\omega_{i_m}\otimes\cdots\otimes\omega_{i_1})(\gamma)=
%\gamma\mapsto \sum_{I}c_I
\int_{0<t_1< \cdots <t_m<1}
\omega_{i_m}({\gamma(t_m)})\cdot
\omega_{i_{m-1}}({\gamma(t_{m-1})})\cdot\cdots
\omega_{i_1}({\gamma(t_1)}).
\end{equation*}
%for $\omega_{i_m},\dots,\omega_{i_1}\in V$.
Here, $\pi_1(\X(\C);0,z)$ is the set of homotopy paths $\gamma_z$ from $\vec{01}$ to $z$.
%It is an inclusion and 
Actually, $\rho$  induces an isomorphism of Hopf algebras between
$TV$ and the space $\mathrm{Im} \rho$ of iterated integrals over $\X$ due to Chen's theory (cf. \cite{C}).

Since \eqref{eq:20210606} is regarded as the equality
\begin{align*}\label{eq:20210606}
&\left[G_{\vec{0 1}}(Y-X,-Y)(\gamma_z)\cdot
\begin{pmatrix}
1 & 1 \\
\frac{-\aaa}{\bbb} & -1
\end{pmatrix}\right]_{11} 
%\\ \notag
%&\qquad \qquad 
=
\exp\left\{\log(\gamma_{z})\aaa\right\}\cdot
\iota\left(\left[G_{\vec{01}}(X,-Y)(\gamma_z)\cdot
\begin{pmatrix}
1 & 1 \\
0 & \frac{\ppp}{\bbb}
\end{pmatrix}\right]_{11}\right)
\end{align*}
in 
$\mathrm{Map}\left(\pi_1(\X(\C);0,z),\ \C\right)\widehat\otimes
\C\langle\langle e_0,e_1\rangle\rangle$,
%it is also regarded as 
it yields an equality in 
$TV\widehat\otimes
\C\langle\langle e_0,e_1\rangle\rangle$.

%As is explained in ....,  iterated integrations induce an inclusion
%$$
%\rho:H^0\bar B(\Omega^\ast_\mathrm{DR}(\X))
%%\simeq I^{z}_{0}(X)
%\hookrightarrow
%\mathrm{Map}\left(\mathcal P_{0}^{z}(\X(\C)), \C
%%\C\langle\langle e_0,e_1\rangle\rangle
%\right).
%$$
The above  $G_{\vec{01}}(e_0,e_1)(z)=G_{\vec{01}}(e_0,e_1)(\gamma_z)$ corresponds to the element 
$G(e_0,e_1)$
in the $\Q$-structure
%in the completed tensor product
$TV %H^0\bar B(\Omega^\ast_\mathrm{DR}(\X)) 
\hat\otimes
\Q\langle\langle e_0,e_1\rangle\rangle
$
given by
$$
G(e_0,e_1):=\sum_{W:\text{ word}} \Omega_W\otimes W
$$
where for each word $W$,
we mean $\Omega_W$  to be an element in $TV$ obtained by substituting
$[\frac{dz}{z}]$ (resp. $[\frac{dz}{z-1}]$) for
$e_0$ (resp. $e_1$) 
in $V\subset TV$.  %%EDITOR'S NOTE: Please ensure that the intended
                   %%meaning has been maintained in this edit.
Whence by \eqref{eq:20210606}, we have
$$
\left[G(Y-X,-Y)\cdot
\begin{pmatrix}
1 & 1 \\
\frac{-\aaa}{\bbb} & -1
\end{pmatrix}\right]_{11} 
=\exp\left\{[\frac{dz}{z}]\aaa\right\}\cdot
\iota\left(\left[G(X,-Y)\cdot
\begin{pmatrix}
1 & 1 \\
0 & \frac{\ppp}{\bbb}
\end{pmatrix}\right]_{11}\right)
$$
in 
%$H^0\bar B(\Omega^\ast_\mathrm{DR}(\X)) 
$TV\hat\otimes
\Q[[\aaa,\bbb,\ccc-1]]$.

Assume that $g=1+\sum_WI(W)W$ is any group-like series in $\Q_\ell\langle\langle e_0,e_1\rangle\rangle$.
Since $g$ is group-like, $I(W)$s satisfy the shuffle product.
We have a shuffle algebra homomorphism 
$\ev_g: TV %H^0\bar B(\Omega^\ast_\mathrm{DR}(\X))
\to \Q_\ell$
sending $\Omega_W$ to $I(W)$.
By applying $\ev_g$ to the above equality,
we obtain
%So the above equality implies
$$
\left[g(Y-X,-Y)\cdot
\begin{pmatrix}
1 & 1 \\
\frac{-\aaa}{\bbb} & -1
\end{pmatrix}\right]_{11} 
=\exp\left\{I(e_0)\aaa\right\}\cdot
\iota\left(\left[g(X,-Y)\cdot
\begin{pmatrix}
1 & 1 \\
0 & \frac{\ppp}{\bbb}
\end{pmatrix}\right]_{11}\right)
$$
in $\Q_\ell[[\aaa,\bbb,\ccc-1]]$.
This is how our claim is proved.

\bigskip
{\it The second equality:} 
%Let $\iota'$ be the automorphism of $\Q_\ell[[\aaa,\bbb,\ccc-1]]$
%sending $\aaa$, $\bbb$, $\ccc$ to
%$\aaa$, $\aaa+1-\ccc$, $\aaa+\bbb+1-\ccc$ respectively.
By \eqref{eq:PhiG10}, \eqref{eq:PhiG1infty} and \eqref{eq:Phi10}, 
it is sufficient to show that the equation
%we have
%\begin{align*}
%\V_{\vec{1\infty}}^{\varphi}(\sigma)(z)
%&=
%G^\varphi_{\vec{1\infty}}(X,-Y)(\sigma)(z)\cdot
%\begin{pmatrix}
%1 & 0 \\
%\frac{-\aaa}{\qqq} & \frac{1-\qqq}{\bbb}
%\end{pmatrix}
%=
%G^\varphi_{\vec{10}}(Y-X,-Y)(\sigma)(\frac{1}{z})\cdot
%\begin{pmatrix}
%1 & 0 \\
%\frac{-\aaa}{\qqq} & \frac{1-\qqq}{\bbb}
%\end{pmatrix} \\
%&=
%G^\varphi_{\vec{10}}(X,-Y)(\sigma)(\frac{1}{z})\cdot
%\begin{pmatrix}
%1 & 0 \\
%\frac{-\aaa}{\qqq} & \frac{\qqq-1}{\bbb}
%\end{pmatrix}
%\begin{pmatrix}
%1 & 0 \\
%0 & -1
%\end{pmatrix}
%=
%\V_{\vec{10}}^{\varphi}(\sigma)(\frac{1}{z})\cdot
%\begin{pmatrix}
%1 & 0 \\
%0 & -1
%\end{pmatrix}.
%\end{align*}
%By comparing the $(1,1)$ entry, we get the equality.
\begin{align*}
&\left[G_{\vec{10}}^\varphi(Y-X,-Y)(\sigma)(\frac{1}{z})\cdot
\begin{pmatrix}
1 & 0 \\
\frac{-\aaa}{\qqq} & \frac{1-\qqq}{\bbb}
\end{pmatrix}\right]_{11} \\
&\qquad \qquad =
\exp\left\{-\rho_{{z}}(\sigma)\aaa\right\}\cdot
\iota\left(\left[G_{\vec{10}}^\varphi(X,-Y)(\sigma)(\frac{1}{z})\cdot
\begin{pmatrix}
1 & 1 \\
\frac{-\aaa}{\qqq} & \frac{\qqq-1}{\bbb}
\end{pmatrix}\right]_{11}\right)
\end{align*}
holds in $\PP[\qqq^{-1}]$.

In the complex case, 
we have
$G_{\vec{10}}(Y-X,-Y)(\frac{1}{z})=G_{\vec{1\infty}}(X,-Y)(z)$.
Thus, by the formula for
$\begin{pmatrix} 1, & 0 \end{pmatrix}\cdot \V_{\vec{1\infty}}(z)$
in \S \ref{sec:complex case}, we  have
$$
\left[G_{\vec{10}}(Y-X,-Y)(\frac{1}{z})\cdot
\begin{pmatrix}
1 & 0 \\
\frac{-\aaa}{\qqq} & \frac{1-\qqq}{\bbb}
\end{pmatrix}\right]_{11}
=\exp\left\{-\log(z)\aaa\right\}\cdot
\HG{\aaa,\aaa-\ccc+1}{\aaa+\bbb+1-\ccc}{1-\frac{1}{z}}.
$$
While by \cite{O} (34) and (70),
\footnote{
It looks there is an error in the matrix on the equation \cite{O}  (70). The $(2,2)$ entry should be $\frac{\beta}{\alpha+\beta-\gamma}$ instead of$\frac{1}{\alpha+\beta-\gamma}$.
}%%EDITOR'S NOTE: Please ensure that all of the text that is serving
 %%as a placeholder or note is appropriately replaced or removed.
we have
$$
\left[G_{\vec{10}}(X,-Y)(z)\cdot
\begin{pmatrix}
1 & 0 \\
\frac{-\aaa}{\qqq} & \frac{\qqq-1}{\bbb}
\end{pmatrix}\right]_{11}
=\HG{\aaa,\bbb}{\aaa+\bbb+1-\ccc}{1-z}.
$$
Thus, we obtain
{\small
$$
\left[G_{\vec{10}}(Y-X,-Y)(\frac{1}{z})\cdot
\begin{pmatrix}
1 & 0 \\
\frac{-\aaa}{\qqq} & \frac{1-\qqq}{\bbb}
\end{pmatrix}\right]_{11} 
%\\
%&\qquad \qquad 
=
\exp\left\{-\log(z)\aaa\right\}\cdot
\iota\left(\left[G_{\vec{10}}(X,-Y)(\frac{1}{z})\cdot
\begin{pmatrix}
1 & 1 \\
\frac{-\aaa}{\qqq} & \frac{\qqq-1}{\bbb}
\end{pmatrix}\right]_{11}\right)
$$
}\par
\noindent where $\iota$ is an extension our previously introduced $\iota$ to $\PP[\qqq^{-1}]$ (N.B. $\iota(\qqq)=\qqq$).
By the same arguments as for the proof of the first equality,
we obtain the claim.

\bigskip
{\it The third equality:}
By \eqref{eq:PhiGinfty0}, \eqref{eq:PhiG0infty} and \eqref{eq:Phi0infty}, 
it is sufficient to prove that the equation
\begin{align*}
&\left[G_{\vec{0\infty}}^\varphi(Y-X,-Y)(\sigma)(\frac{1}{z})\cdot
\begin{pmatrix}
1 & 1 \\
\frac{-\aaa}{\bbb} & -1
\end{pmatrix}\right]_{11} \\
&\qquad \qquad =
\exp\left\{-\rho_{(1-z)}(\sigma)\aaa\right\}\cdot
\varsigma\left(\left[G_{\vec{0\infty}}^\varphi(X,-Y)(\sigma)(\frac{1}{z})\cdot
\begin{pmatrix}
1 & 1 \\
0 & \frac{\ppp}{\bbb}
\end{pmatrix}\right]_{11}\right)
\end{align*}
holds in $\bbb^{-1}\PP$.
Here, $\varsigma$ means the automorphism of $\PP$
sending $\aaa$, $\bbb$, and $\ccc$ to
$\aaa$, $\aaa+1-\ccc$, and $\aaa-\bbb+1$, respectively.

In the complex case, similarly, we have
$$
\left[G_{\vec{0\infty}}(Y-X,-Y)(\frac{1}{z})\cdot
\begin{pmatrix}
1 & 1 \\
\frac{-\aaa}{\bbb} & -1
\end{pmatrix}\right]_{11}
=\exp\{-\log(1-z)\aaa\}\cdot
\HG{\aaa,\ccc-\bbb}{\aaa-\bbb+1}{\frac{1}{1-z}},
$$
$$
\left[G_{\vec{0\infty}}(X,-Y)(\frac{1}{z})\cdot
\begin{pmatrix}
1 & 1 \\
0 & \frac{\ppp}{\bbb}
\end{pmatrix}\right]_{11}
=\exp\{-\log(\frac{z-1}{z})\aaa\}\cdot
\HG{\aaa,\ccc-\bbb}{\ccc}{\frac{1}{1-z}}
$$
by the formulae for
$\begin{pmatrix} 1, & 0 \end{pmatrix}\cdot \V_{\vec{\infty0}}(z)$ and 
$\begin{pmatrix} 1, & 0 \end{pmatrix}\cdot \V_{\vec{0\infty}}(z)$
in \S \ref{sec:complex case}.
Again, the same arguments follow our claim.
\end{proof}

A formal version of our $\ell$-adic Euler transformation formula is given as follows:

\begin{thm}\label{thm: formal Euler transformation}
The equality
\begin{equation}
\HGdag{\aaa',\bbb'}{\ccc'}{z}(\sigma)=%(1-z)^{c-a-b}
\exp\left\{({\ccc'}-{\aaa'}-{\bbb'}){\rho_{1-z}}(\sigma)\right\}\cdot
\HGdag{\ccc'-\aaa',\ccc'-\bbb'}{\ccc'}{z}(\sigma)
\end{equation}
holds in $\ccc'^{-1}\Q_\ell[[\aaa',\bbb',\ccc']]$.
\end{thm}

\begin{proof}
By \eqref{eq:Phi1infty}
and the reparametrization 
$\aaa'=\aaa$, $\bbb'=\aaa+1-\ccc$, $\ccc'=\aaa+\bbb+1-\ccc$ and $z=1-\frac{1}{w}$,
the left-hand side is calculated to be
$$
%\HG{\aaa',\bbb'}{\ccc'}{z}(\sigma)=
\HGdag{\aaa,\aaa+1-\ccc}{\aaa+\bbb-\ccc+1}{1-\frac{1}{w}}(\sigma)=
\exp\left\{\rho_{w}(\sigma)\aaa\right\}\cdot
[\V_{\vec{1\infty}}(\sigma)(w)]_{11}
$$
for an even unitary associator $\varphi$.
The right-hand side is calculated to be
\begin{align*}
\exp\left\{\rho_{w}(\sigma)(\aaa-\bbb)\right\}\cdot
&\HGdag{\bbb,\bbb+1-\ccc}{\aaa+\bbb-\ccc+1}{1-\frac{1}{w}}(\sigma) \\
&=\exp\left\{\rho_{w}(\sigma)\aaa\right\}\cdot
\sw_{\aaa,\bbb}\left([\V_{\vec{1\infty}}(\sigma)(w)]_{11}\right)
\end{align*}
where $\sw_{\aaa,\bbb}$ means the automorphism of $\PP[\frac{1}{\qqq}]$
switching  $\aaa$ and $\bbb$.
Hence, it is sufficient to show that 
$[\V_{\vec{1\infty}}(\sigma)(z)]_{11}$
is invariant under the switch
$\sw_{\aaa,\bbb}$.
The matrix is calculated to be
$$
\V_{\vec{1\infty}}(\sigma)(z)=
\left[
G_{\vec{1\infty}}^{\varphi}(X,-Y)(\sigma)(z)\cdot
\begin{pmatrix}
1 & 1 \\
0 & \frac{\ppp}{\bbb}
\end{pmatrix}
\right]
\cdot
\begin{pmatrix}
\frac{\aaa\bbb+\ppp\qqq}{\ppp\qqq} & \frac{\qqq-1}{\ppp} \\
\frac{-\aaa\bbb}{\ppp\qqq} & \frac{1-\qqq}{\ppp}
\end{pmatrix}.
$$
It is evident that the  last matrix
is invariant under the switch $\sw_{\aaa,\bbb}$.
The first row of the product 
$
\left[
G_{\vec{1\infty}}^{\varphi}(X,-Y)(\sigma)(z)\cdot
\begin{pmatrix}
1 & 1 \\
0 & \frac{\ppp}{\bbb}
\end{pmatrix}
\right]
$
%the first and the second matrix 
is  also invariant under the switch $\sw_{\aaa,\bbb}$.
This is because the $(1,1)$ entry is calculated to be
$$
1+ \aaa\bbb\sum_{\substack{k,n,s>0 \\k>n+s, n \geqslant s}}
h_0(k,n,s)\ppp^{k-n-s}\qqq^{n-s}
(\aaa\bbb+\ppp\qqq)^{s-1}
$$
by \cite{O} (65)
and the $(1,2)$ entry is calculated to be
$$
\exp\{\rho_{1-\frac{1}{z}}(\sigma)\ppp\}
\cdot\bigl\{
1+ (\aaa\bbb+\ppp\qqq)\sum_{\substack{k,n,s>0 \\k>n+s, n \geqslant s}}
h_0(k,n,s) (-\ppp)^{k-n-s}\qqq^{n-s}
(\aaa\bbb)^{s-1}
\bigr\}
$$
by \cite{O} (66) and its following remark.
Here, $h_0(k,n,s)$ is given by
$$
\sum_{\substack{ W\in{e_0{U\frak f_2}e_1}:\text{ word} \\ \wt(W)=k, \ \mathrm{dp}(W)=n, \ \mathrm{ht}(W)=s}}
(-1)^{n}( G_{\vec{1\infty}}^{\varphi}(e_0,e_1)(\sigma)(z)
\bigm| W).
$$
%\Add{(Explain more!)}
Whence we obtain the equality.
\end{proof}

%A formal version of our $\ell$-adic Pfaff transformation formula is given as follows:
%
%\begin{thm}\label{thm: formal Pfaff transformation}
%The equality
%\begin{equation}
%\HG{\aaa',\bbb'}{\ccc'}{z}^\dagger(\sigma)=%(1-z)^{c-a-b}
%\exp\left\{-{\aaa'}\cdot{\rho_{1-z}}(\sigma)\right\}\cdot
%\HG{\ccc'-\bbb',\aaa'}{\ccc'}{\frac{z}{z-1}}^\dagger(\sigma)
%\end{equation}
%holds in $\ccc'^{-1}\Q_\ell[[\aaa',\bbb',\ccc']]$.
%\end{thm}
%
%\begin{proof}
%\end{proof}

%Before we go to prove
%Theorem \ref{thm: l-adic Euler's transformation formula},
%we show that our $\HG{a,b}{c}{z}{(\sigma)}$ is defined for
%$(a,b,c)\in{\mathfrak D}$.

The following proposition  enables us to extend  our $\ell$-adic hypergeometric function
$\HG{a,b}{c}{z}{(\sigma)}$ 
defined when $|a|_\ell, |b|_\ell, |c-1|_\ell<1$
to the parameter $(a,b,c)$ in ${\mathfrak D}$
by
\begin{align*}
&\HG{a,b}{c}{z}{(\sigma)}:=\HGdag{a,b}{c}{z}{(\sigma)} \quad \text{when} \quad
|a|_\ell, |b|_\ell, |c|_\ell<1 \text{ with } c\neq 0,
\\
&\HG{a,b}{c}{z}{(\sigma)}:=\HGddag{a,b}{c}{z}{(\sigma)}\quad \text{when} \quad
|a|_\ell, |b-1|_\ell, |c-1|_\ell<1.
\end{align*}

\begin{prop}\label{prop:enlarge}
(1).
When $|a'|_\ell, |b'|_\ell, |c'|_\ell<1$ with $c'\neq 0$, the evaluation of \eqref{eq:2F1prime}
to $a',b',c'$ converges.

(2).
When $|a''|_\ell, |b''-1|_\ell, |c''-1|_\ell<1$, the evaluation of 
\eqref{eq:2F1primeprime}
to $a'',b'',c''$ converges.
\end{prop}

\begin{proof}
It is sufficient to show that both
$G_{\vec{10}}^\varphi(X,-Y)(\sigma)(z)$ and
$G_{\vec{0\infty}}^\varphi(X,-Y)(\sigma)(z)$ 
converge when \eqref{eq:ABC condition} holds
by \eqref{eq:Phi10} and \eqref{eq:Phi0infty}.
They actually converge 
by \eqref{eq:cont bmod ell}, \eqref{eq:G10=f} and \eqref{eq:G0infty=f}.
\end{proof}
%\smallskip
The $\ell$-adic analogue of
Euler's transformation theorem can be derived as follows:
\smallskip

\noindent
{\bf Proof of
Theorem \ref{thm: l-adic Euler's transformation formula}.}
The claim follows from
Theorem  \ref{thm: formal Euler transformation} and 
Proposition \ref{prop:enlarge}. (1)
because we have $\rho_{1-z}(\sigma)\in\Z_\ell$.
\qed

\begin{rem}
Our arguments for deducing results in the $\ell$-adic situation  from those in the complex case
that are observed in the proof of Theorem \ref{thm: formal Euler transformation} 
allow us to show the following $\ell$-adic analogues of 
the six formulae \eqref{eqD} of Kummer's solutions:
%in \S \ref{sec:complex case}:
{\footnotesize
\begin{align*}
%\begin{pmatrix} 1, & 0 \end{pmatrix}
(1,0)\cdot \V_{\vec{01}}(\sigma)(z) &=
\begin{pmatrix} \HG{a,b}{c}{z}(\sigma), & \exp\{(1-c)\rho_z(\sigma)\}\cdot\HG{b+1-c,a+1-c}{2-c}{z}(\sigma)\end{pmatrix}, \\
%\begin{pmatrix} 1, & 0 \end{pmatrix}
(1,0)\cdot \V_{\vec{10}}(\sigma)(z)
&=
\begin{pmatrix} \HGdag{b,a}{a+b+1-c}{1-z}(\sigma), & \exp\{(c-a-b)\rho_{1-z}(\sigma)\}\cdot\HGdag{c-a,c-b}{c-a-b+1}{1-z}(\sigma)\end{pmatrix}, \\
%\begin{pmatrix} 1, & 0 \end{pmatrix}
(1,0)\cdot \V_{\vec{1\infty}}(\sigma)(z)
&=
\Bigl(
%\begin{pmatrix} 
\exp\{-a\rho_z(\sigma)\}\cdot\HGdag{a,a+1-c}{a+b-c+1}{1-\frac{1}{z}}(\sigma), \\
&\qquad\qquad \exp\{(b-c)\rho_z(\sigma)\}\cdot \exp\{(c-a-b)\rho_{z-1}(\sigma)\}\cdot\HGdag{1-b,c-b}{1-a-b+c}{1-\frac{1}{z}}(\sigma)
%\end{pmatrix}
\Bigr), \\
%\begin{pmatrix} 1, & 0 \end{pmatrix}
(1,0)\cdot \V_{\vec{\infty 1}}(\sigma)(z)
&=
\begin{pmatrix} \exp\{-a\rho_z(\sigma)\}\cdot\HG{a,a+1-c}{a-b+1}{\frac{1}{z}}(\sigma), & \exp\{-b\rho_z(\sigma)\}\cdot\HG{b+1-c,b}{b-a+1}{\frac{1}{z}}(\sigma)
\end{pmatrix}, \\
%\begin{pmatrix} 1, & 0 \end{pmatrix}
(1,0)\cdot \V_{\vec{\infty 0}}(\sigma)(z)
&=
\begin{pmatrix} \exp\{-a\rho_{1-z}(\sigma)\}\cdot\HGddag{a,c-b}{a-b+1}{\frac{1}{1-z}}(\sigma),& \exp\{-b\rho_{1-z}(\sigma)\}\cdot\HGddag{c-a,b}{1-a+b}{\frac{1}{1-z}}(\sigma)
\end{pmatrix},\\
%\begin{pmatrix} 1, & 0 \end{pmatrix}
(1,0)\cdot \V_{\vec{0\infty}}(\sigma)(z)
&=
%\begin{pmatrix} 
\Bigl(\exp\{-a\rho_{1-z}(\sigma)\}\cdot\HGddag{a,c-b}{c}{\frac{z}{z-1}}(\sigma), \\
&\qquad\qquad \exp\{-a\rho_{1-z}(\sigma)\}\exp\{(1-c)\rho_{\frac{z}{z-1}}(\sigma)\}\cdot
\HGddag{1-b,a-c+1}{2-c}{\frac{z}{z-1}}(\sigma)\Bigr).
%\end{pmatrix},
\end{align*}
}
\end{rem}

\begin{rem}
Finite field analogues of hypergeometric functions  have been discussed in 
the literature; see \cite{FLRST, Greene, Katz, Otsubo}, etc.
Since they are related to the  trace of Frobenius of certain $\ell$-adic Galois representations and our $\ell$-adic function is constructed from a  Galois representation,
it would be worthwhile to identify any
%there might be a 
relationship between their hypergeometric functions and ours.
\end{rem}

%%%%%%%%%%%%%%%%%%%%%%%%%%%%%%%%%%%%%%%%%%%%%%%%%%%%%%%%%%%%%%%%%%%%%%
%\bigskip
{\it Acknowledgments.} The author has been supported by JSPS KAKENHI JP18H01110,
JP20H00115 and JP21H00969.
He is grateful to 
Benjamin Enriquez,
Hiroaki Nakamura, 
Noriyuki Otsubo,
and
Seidai Yasuda
who suggested various studies %explained  him accounts of 
and %indicated him 
directions related to this work
and  Ryotaro Harada who carefully read
an earlier version of the paper.

%%%%%%%%%%%%%%%%%%%%%%%%%%%%%%%%%%%%%%%%%%%%%%%%%%%%%%%%%%%%%%%%%%
\appendix
\section{On  $\ev_{(X,-Y)}(\varphi(e_\infty,e_1))$}
%Computations of $\V_{\vec{\infty 1}}^{\varphi}(\sigma)(z)$  and $\V_{\vec{\infty 0}}^{\varphi}(\sigma)(z)$}
\label{appendix A}

We prove an auxiliary lemma that is needed for the  proof of Proposition \ref{prop:even associator free}.

\begin{lem}\label{lemma in appendix}
The matrix 
$\ev_{(X,-Y)}(\varphi(e_\infty,e_1))$ in $\Mat_2(\K[[\aaa,\bbb,\ccc-1]])$
%denoted by $N_+$, 
is independent of any choice of even unitary associators $\varphi$.
\end{lem}

\begin{proof}
For any group-like series $\varphi(e_0,e_1)\in\K\langle\langle e_0,e_1\rangle\rangle$,
write
$P_\varphi:=\ev_{(X,-Y)}\left(\varphi(e_0,e_1)\right)\cdot
\begin{pmatrix}
1 & 1 \\
0 & \frac{\ppp}{\bbb}
\end{pmatrix}$
and
$Q_\varphi=
\ev_{(X,-Y)}\left(\varphi(e_\infty,e_1)\right)\cdot
\begin{pmatrix}
1 & \Add{1} \\
\frac{-\aaa}{\bbb} & -1
\end{pmatrix}
$
%in $\GL_2(\PP[\frac{1}{\bbb}])$.
in $\GL_2(\K[[\aaa,\bbb,\ccc-1]][\frac{1}{\bbb}])$.

Specifically, when $\varphi=G_{\vec{01}}(e_0,e_1)(z)$,
by \eqref{eqB}, \eqref{eqC} and \eqref{eqD}, we have
{%\small
\begin{align*}
P_\varphi&=
\ev_{(X,-Y)}\left(G_{\vec{01}}(e_0,e_1)(z)\right)\cdot
\begin{pmatrix}
1 & 1 \\
0 & \frac{\ppp}{\bbb}
\end{pmatrix}
=\V_{\vec{01}}(z)\\
&=
\begin{pmatrix} 
\HG{\aaa,\bbb}{\ccc}{z}, & z^{\ppp}\HG{\bbb+1-\ccc,\aaa+1-\ccc}{2-\ccc}{z} \\
\frac{z}{\bbb}\dHG{\aaa,\bbb}{\ccc}{z}, & \frac{\ppp}{\bbb}z^{\ppp}\HG{\bbb+1-\ccc,\aaa+1-\ccc}{2-\ccc}{z}+
z^\bbb\frac{z}{\bbb}\dHG{\bbb+1-\ccc,\aaa+1-\ccc}{2-\ccc}{z}
\end{pmatrix},
\end{align*}
}
{\small
\begin{align*}
&Q_\varphi=
\ev_{(X,-Y)}\left(G_{\vec{01}}(e_\infty,e_1)(z)\right)\cdot
%G_{\vec{\infty 1}}(X_0,-Y_0)(\frac{1}{z})\cdot
\begin{pmatrix}
1 & \Add{1} \\
\frac{-\aaa}{\bbb} & -1
\end{pmatrix}
=\V_{\vec{\infty 1}}(\frac{1}{z})\\
&=
\begin{pmatrix} 
z^{\aaa}\HG{\aaa,\aaa+1-\ccc}{\aaa-\bbb+1}{z}, & 
z^{\bbb}\HG{\bbb+1-\ccc,\bbb}{\bbb-\aaa+1}{z}\\
-\frac{\aaa}{\bbb}z^{\aaa}\HG{\aaa,\aaa+1-\ccc}{\aaa-\bbb+1}{z}
-z^\aaa\frac{z}{\bbb}\dHG{\aaa,\aaa+1-\ccc}{\aaa-\bbb+1}{z}, &
-z^\bbb\HG{\bbb+1-\ccc,\bbb}{\bbb-\aaa+1}{z}-z^\bbb\frac{z}{\bbb}\dHG{\bbb+1-\ccc,\bbb}{\bbb-\aaa+1}{z}
\end{pmatrix}
\end{align*}
}
\par
\noindent
with $z^\aaa:=\exp(\log z\cdot\aaa)$, 
$z^\bbb:=\exp(\log z\cdot\bbb)$ and $z^\ppp:=\exp(\log z\cdot\ppp)$.
Then, by the same arguments as in the  proof of Proposition \ref{three compatibilities},
we  deduce the following validity for each entry
from the above two expressions in  the complex case
when $\varphi$ is commutator group-like:
\begin{align*}
[Q_\varphi]_{11}&=\iota([P_\varphi]_{11}),\\
[Q_\varphi]_{12}&=\iota([P_\varphi]_{12}), \\
[Q_\varphi]_{21}&=-\frac{\aaa}{\bbb}\iota([P_\varphi]_{11})-\frac{1}{\bbb}\iota(\bbb[P_\varphi]_{21}),\\
[Q_\varphi]_{22}&=-\iota([P_\varphi]_{12})-\frac{1}{\bbb}\iota(\bbb[P_\varphi]_{22}-{\ppp}[P_\varphi]_{12})
\end{align*}
where $\iota$ is the map of \eqref{eq:iota-map}.

Assume that $\varphi$ is an even unitary  associator.
Then, we have
$P_\varphi=\ev_{(X,-Y)}(\varphi)
=M_+\cdot
\begin{pmatrix}
1 & 1 \\
0 & \frac{\ppp}{\bbb}
\end{pmatrix}$
by Theorem \ref{thm: varphi=M}.
Whence $P_\varphi$ is free from any choice of even unitary associators
by Theorem \ref{thm:even unitary associator=M+0}.
By the above four equalities, we see that $Q_\varphi$ is so.
By
$Q_\varphi=
\ev_{(X,-Y)}\left(\varphi(e_\infty,e_1)\right)\cdot
\begin{pmatrix}
1 & \Add{1} \\
\frac{-\aaa}{\bbb} & -1
\end{pmatrix}
$,
we learn that $\ev_{(X,-Y)}\left(\varphi(e_\infty,e_1)\right)$
is  free from any choice of even unitary associators.
\end{proof}

\end{document}